\documentclass[a4paper,12pt,article]{amsart}
\usepackage[utf8]{inputenc}
\usepackage[T1]{fontenc}
\usepackage[UKenglish]{babel}
\usepackage[a4paper,margin=28mm]{geometry}
\usepackage{verbatim}
\allowdisplaybreaks[4]
\usepackage{times}
\usepackage{dsfont,mathrsfs}
\usepackage{amsmath}
\usepackage{amsthm}
\usepackage{amssymb}
\usepackage{amsfonts}
\usepackage{latexsym}
\usepackage{booktabs}
\usepackage[colorlinks,pagebackref]{hyperref}
\usepackage{xcolor}
\usepackage{natbib}
\setcitestyle{authoryear,round}
\newtheorem{theorem}{Theorem}[section]
\newtheorem{thm}[theorem]{Theorem}

\newtheorem{lemma}[theorem]{Lemma}
\newtheorem{lem}[theorem]{Lemma}
\newtheorem{proposition}[theorem]{Proposition}

\newtheorem{corollary}[theorem]{Corollary}

\theoremstyle{definition}
\newtheorem{definition}[theorem]{Definition}
\newtheorem{defn}[theorem]{Definition}

\theoremstyle{remark}
\newtheorem{remark}[theorem]{Remark}
\newtheorem{rem}[theorem]{Remark}
\numberwithin{equation}{section}

 \DeclareMathAlphabet{\mathpzc}{OT1}{pzc}{m}{it}

  \newcommand{\dif}{\mathrm{d}}
 
 \newcommand{\A}{\mathcal{A}}

 \newcommand{\HH}{\mathbb{\mathbb{H}}}
 \newcommand{\F}{\mathcal{F}}

 \newcommand{\cov}{\mathrm{cov}}       
 \newcommand{\E}{\mathbb{E}}            
 
 \newcommand{\e}{\varepsilon}

 \newcommand{\Ll}{\langle}
 \newcommand{\Rr}{\rangle}

 \newcommand{\N}{\mathbb{N}}
 \newcommand{\R}{\mathbb{R}}
 
 \newcommand{\PP}{\mathbb{P}}

 \newcommand{\Be}{\begin{equation}}
 \newcommand{\Ee}{\end{equation}}
  \newcommand{\Bes}{\begin{equation*}}
 \newcommand{\Ees}{\end{equation*}}
  \newcommand{\Bey}{\begin{eqnarray}}
 \newcommand{\Eey}{\end{eqnarray}}
 \newcommand{\Beys}{\begin{eqnarray*}}
 \newcommand{\Eeys}{\end{eqnarray*}}
 \newcommand{\BT}{\begin{thm}}
 \newcommand{\ET}{\end{thm}}
 \newcommand{\Bp}{\begin{proof}}
 \newcommand{\Ep}{\end{proof}}
 \newcommand{\BL}{\begin{lem}}
 \newcommand{\EL}{\end{lem}}
 \newcommand{\BP}{\begin{proposition}}
 \newcommand{\EP}{\end{proposition}}
 \newcommand{\BC}{\begin{corollary}}
 \newcommand{\EC}{\end{corollary}}
 \newcommand{\BR}{\begin{rem}}
 \newcommand{\ER}{\end{rem}}
 \newcommand{\BD}{\begin{defn}}
 \newcommand{\ED}{\end{defn}}
 \newcommand{\BI}{\begin{itemize}}
 \newcommand{\EI}{\end{itemize}}

\usepackage{marginnote}
\begin{document}

\title[ ]
{ Almost sure invariance principle of $\beta-$mixing time series in Hilbert space}

\author[J. Lu]{Jianya Lu}
\address[J. Lu]{Department of Mathematics, University of Macau; Department of Mathematics, University of Essex}
\email{jylu2018@gmail.com}

\author[W. B. Wu]{Wei Biao Wu}
\address[W. B. Wu]{Department of Statistics, University of Chicago}
\email{wbwu@galton.uchicago.edu}

\author[Z. Xiao]{Zhijie Xiao}
\address[Z. Xiao]{Department of Economics, Boston College}
\email{zhijie.xiao@bc.edu}

\author[L.~Xu]{Lihu Xu}
\address[L.~Xu]{Department of Mathematics, Faculty of Science and Technology, University of Macau, Macau S.A.R., China }
\email{lihuxu@um.edu.mo}

\maketitle

\begin{abstract}
Inspired by \citet{Berkes14} and \citet{Wu07}, we  prove an almost sure invariance principle for stationary $\beta-$mixing stochastic processes defined on Hilbert space. Our result can be applied to Markov chain satisfying Meyn-Tweedie type Lyapunov condition and thus generalises the contraction condition in \citet[Example 2.2]{Berkes14}. 
We prove our main theorem by the big and small blocks technique and an embedding result in \citet{gotze2011estimates}. Our result is further applied to the ergodic Markov chain and functional autoregressive processes.

	
\end{abstract}

\section{Introduction}\label{sec:intro}
Let $\mathbb H$ be a separable Hilbert space with orthonormal basis $\{{\boldsymbol e}_{k}\}_{k \in \N}$, denote by $\Ll\cdot,\cdot\Rr$ and  $\|\cdot\|$ the associated inner product and norm respectively. For any ${\boldsymbol x} \in \mathbb H$, there exists {   an} unique representation ${\boldsymbol x} =\sum_{k=1}^{\infty} x_{k} {\boldsymbol e}_{k}$ with $x_{k} \in \R$ for $k \ge 1$ and $\sum_{k=1}^{\infty} x^{2}_{k}<\infty$, and thus we can represent ${\boldsymbol x}$ by the sequence $(x_{1},...,x_{k},...)^T$, where $T$ is the transpose operator.   For positive integer $d$ and any ${\boldsymbol x}\in\HH$, $\HH$ can be decomposed into 
$$\HH=\HH_{\le d}+\HH_{> d}$$
with element $(x_1,...,x_d)^T\in \HH_{\le d}$ and $(x_{d+1},x_{d+2}...)^T\in \HH_{> d}$.

Let ${\boldsymbol X}$ and ${\boldsymbol Y}$ be two $\mathbb H$-valued random variables, denote by $\PP_{{\boldsymbol X} \times {\boldsymbol Y}}$ the joint probability of $({\boldsymbol X},{\boldsymbol Y})$ and by $\PP_{\boldsymbol X}$, $\PP_{\boldsymbol Y}$ the probabilities of ${\boldsymbol X}$ and ${\boldsymbol Y}$ respectively, define
\begin{eqnarray*}
\beta\left({\boldsymbol X},{\boldsymbol Y}\right):=\|\PP_{{\boldsymbol X}\times {\boldsymbol Y}}-\PP_{\boldsymbol X}\times\PP_{\boldsymbol Y}\|_{\mathrm{TV}},
\end{eqnarray*}
where $\|.\|_{\mathrm{TV}}$ is the total variation norm of probability {measures}, i.e.
$$\|\PP_{{\boldsymbol X}\times {\boldsymbol Y}}-\PP_{\boldsymbol X}\times\PP_{\boldsymbol Y}\|_{\mathrm{TV}}=\sup_{f \in C_b(\mathbb H \times \mathbb H): \|f\|_\infty \le 1} |\PP_{{\boldsymbol X}\times {\boldsymbol Y}}(f)-\PP_{\boldsymbol X}\times\PP_{\boldsymbol Y}(f)|. $$
Let $({\boldsymbol X}_{k})_{k \in \N_{0}}$ be an $\mathbb H$-valued time series with $\N_{0}=\{0\} \cup \N$, its $\beta-$mixing coefficient is defined as
\begin{eqnarray}\label{e:mixing}
\beta(n)=\sup_{k\ge1}\beta\left(({\boldsymbol X}_i)_{0\le i\le k}, ({\boldsymbol X}_i)_{i\ge n+k}\right).
\end{eqnarray}
In this paper, we assume that

{\bf (A1)}  $({\boldsymbol X}_{k})_{k \in \N_{0}}$  defined on the probability space $(\Omega,\F,\PP)$ is stationary with {   a} stationary measure $\pi$ and exponentially $\beta$-mixing, i.e., there exist some $C>0$ and $\beta>0$ such that
$$\beta(n) \le C e^{-\beta n}, \ \ \ \ n \in \N.$$

{\bf (A2)} There exists some positive linear operator $\boldsymbol\Gamma: \mathbb H \rightarrow \mathbb H$ such that
$$\lim_{n \rightarrow \infty}\frac{\cov(\sum_{k=0}^{n-1}{\boldsymbol X}_{k})}{n}=\boldsymbol\Gamma,$$
where $\cov(.)$ is the covariance operator of the $\mathbb H$-valued random variable. Moreover, $\boldsymbol \Gamma$ have positive
eigenvalues $\{\lambda_{k}\}_{k \in \N}$  with $\lambda_{1} \ge \lambda_{2} \ge ... > 0$.  There exist $\delta_1{  \ge}\delta_2>1$ and positive constants $C_1$ and $C_2$, such that
$$ C_1k^{-\delta_1}\le\lambda_{k} \le C_2k^{-\delta_2}, \ \ \ \ \ \ k \in \N.$$
\begin{remark}
It is easy to obtain
 $$\boldsymbol\Gamma=\cov({\boldsymbol X}_0)+\sum_{k=1}^{\infty} \big(\cov({\boldsymbol X}_0,{\boldsymbol X}_k)+\cov({\boldsymbol X}_0,{\boldsymbol X}_k)^T\big).$$	
\end{remark}
\begin{remark}
{\bf (A2)} implies that $\boldsymbol \Gamma$ has effective rank and the eigenvalues of covariance operator polynomially decay. Similar conditions can be found in    \cite{reiss2020nonasymptotic}, \citet{lopes2019bootstrapping} and citations therein. 
\end{remark}
\subsection{Motivations and literature review}
Let $(\boldsymbol Y_k)_{k\in\N_0}$ be random variables with $\E\boldsymbol Y_k=0$ and $\E|\boldsymbol Y_k|^p<\infty$, $(\boldsymbol Y_k)_{k\in\N_0}$ satisfies the almost sure invariance principle (ASIP) with rate $r_n$ if there exists, after suitably enlarging the probability space, independent Gaussian random variables  $(\boldsymbol\eta_k)_{k\in\N_0}$ with covariance $\cov(\boldsymbol Y_k)$, that is $\boldsymbol\eta_k\sim\mathcal N(0,\cov(\boldsymbol Y_k))$ such that 
\begin{eqnarray*}
\max_{1\le i\le n}|\sum_{k=1}^{i}\boldsymbol Y_k-\sum_{k=1}^{i}\boldsymbol\eta_k|=o(r_n),\quad \text{a.s.},
\end{eqnarray*}
where $o(\cdot)$ defined as follows, let $\{a_n\}_{n \ge 1}$ and $\{b_n\}_{n \ge 1}$ be two nonnegative real number sequences, if $\lim_{n \rightarrow \infty} \frac{a_n}{b_n}=0$, we write $a_n=o(b_n)$. We further denote $a_n\asymp b_n$ if there exist positive constants $c$ and $C$ such that $ca_n\le b_n\le Ca_n$.

The ASIP was introduced by \citet{strassen1964invariance,strassen1967almost} to prove the functional law of iterated logarithm for independent, identically distributed (i.i.d.) random variables. Besides that, the ASIP implies Donsker's theorem. \citet{komlos1975approximation,komlos1976approximation} considered this problem for i.i.d. random variables with $p-$th moment and got the optimal rate $n^{1/p}$ for $p>2$.

For the ASIP of dependent random variables, \citet{Wu07} used martingale approximation and the {Skorokhod} embedding to prove the ASIP of rate $n^{1/p}(\log n)^{1/2}$ for a class of stationary processes with $2<p<4$ on $1$-dimensional space. However, due to \cite{monrad1991problem} which showed that one can not embed a general $\R^d-$valued martingale in an $\R^d-$valued Gaussian process, it is difficult to extend the martingale embedding method to $d$-dimensional space. \cite{Liu09} used the block technique to get the rate $n^{1/p}$ for $2<p<4$ on $\R^d$ space. Within this framework, \cite{Berkes14} followed the block technique and some skills to construct the independence of the block sum, they got the rate $n^{1/p}$ for $p>2$ on $\R$. For more research about the ASIP on $\R$, we refer the reader to \citet{lu1987strong, merlevede2012strong, cuny2020rates1,cuny2020rates2}.

For the multidimensional invariance principle, \cite{zaitsev1998multidimensional}  obtained optimal rate for the multivariate version of the KMT theorem (see, \cite{komlos1975approximation,komlos1976approximation}),  and then established multivariate versions of Sakhanenko’s theorem, see \cite{zaitsev2001multidimensional,zaitsev2002multidimensional1,zaitsev2002multidimensional2}. We refer the reader to \cite{gotze2009bounds,gotze2010rates} for the follow-up work. These Gaussian approximation results for independent multivariate random variables give a tool to prove the ASIP for multidimensional dependent random variables. \cite{gouezel2010almost} used the block technique and constructed the independence for the block sums, then got the ASIP for multidimensional dynamical system with the rate greater than $n^{1/4+1/(4p-4)}$ for $p>2$ on $\R^d-$space under the mixing condition which is  represented by the characteristic function. Following similar conditions, \cite{hafouta2020almost} got the rate greater than $n^{1/4}$ for uniformly bounded $\phi-$mixing sequence defined on $d$-dimensional space.

The motivations of studying the ASIP of stationary $\beta-$mixing time series are two folds. One is that there have been many ASIP results for dependent $\R$ and $\R^d$ valued random variables, see the references above, whereas there are very few this type of	ASIP results for dependent time series in Hilbert space, see \cite{dedecker2010almost, cuny2014martingale} and the references therein. The other is that  $\beta-$mixing property is closely related to the concept of ergodicity for a stationary Markov chain, which can be verified by the Lyapunov condition, see \cite{Davy1973,MeTw93b}. Our result provides a weaker condition comparing with the condition of \citet[Example 2.2]{Berkes14}, which is satisfied by many examples.


\subsection{Main result and methodology}
 For random variables $\xi$ and $\eta$, let $\xi\overset{\mathscr{D}}{=}\eta$ denote they have the same distribution.  
Our main result is given as follows.
\begin{theorem}\label{thm-asip}
Assume that {\bf (A1)} and {\bf (A2)}  hold  and that the stationary measure $\pi$ has $p$-th moment with $p>2$. Then there exists a probability space on which we can define random variables $\boldsymbol X^*_i$ and  i.i.d. Gaussian random variables ${\boldsymbol \eta}_i$ such that {$(\boldsymbol X_i)_{1\le i\le n}\overset{\mathscr{D}}=(\boldsymbol X^*_i)_{1\le i\le n}$}, $\boldsymbol\eta_i \sim \mathcal N(\boldsymbol 0,\boldsymbol \Gamma)$ and
    \begin{align*}
    \max_{1\le i\le n }\left\|\sum_{j=1}^i (\boldsymbol X^*_j-\pi(\boldsymbol X))-\sum_{j=1}^i\boldsymbol\eta_j\right\|=o(n^{\bar\theta}),\quad \text{a.s.},
    \end{align*}	
     where { $\boldsymbol\Gamma$ defined in {\bf (A2)} and  
     $\bar\theta>\max\big\{\frac{(2p-2)\delta_1+2\delta_2+23p-4}{44p-4+(4p-4)\delta_1+2p\delta_2},\frac{(2p-2)\delta_1+2\delta_2+p(p+4)/2-4}{p(p+2)-4+(4p-4)\delta_1+2p\delta_2}\big\}$.  } 
\end{theorem}
\begin{corollary}\label{coro-asip}
Under the conditions of Theorem \ref{thm-asip}, {for any $\e>0$, there exists a $\bar\delta_{p,\e}$ depends on $p$ and $\e$ and defined in  
\eqref{e:delta} such that as  $\lambda_k\asymp k^{-\delta}$ and  $\delta>\bar\delta_{p,\e}$,} there exists a probability space on which we can define random variables $\boldsymbol X^*_i$ and  i.i.d. Gaussian random variables $\boldsymbol\eta_i$ such that { $(\boldsymbol X_i)_{1\le i\le n}\overset{\mathscr{D}}=(\boldsymbol X^*_i)_{1\le i\le n}$}, $\boldsymbol\eta_i\sim \mathcal N(\boldsymbol 0,\boldsymbol\Gamma)$ and
\begin{align} \label{e:CorollaryPoly}
\max_{1\le i\le n }\left\|\sum_{j=1}^i (\boldsymbol X^*_j-\pi(\boldsymbol X))-\sum_{j=1}^i\boldsymbol\eta_j\right\|=o\big(n^{\frac13+\frac{2}{3(3p-2)}+\e}\big),\quad \text{a.s.}.
\end{align}	

\end{corollary}
\begin{remark}
For random variables defined on $\R^d$  with condition {\bf (A1)} and further assume $\boldsymbol\Gamma$ is a positive semidefinite matrix. { As we shall see in step 1 of the proof of Lemma \ref{lem:Zaitsev}, by the block technique and the invariance principle of i.i.d. random vectors in \citet[Corollary 3]{zaitsev2007estimates}, when $\lambda>\frac14+\frac{1}{4(p-1)}$, the following ASIP holds,}
\begin{align}\label{e:remark}
\max_{1\le i\le n }\left\|\sum_{j=1}^i (\boldsymbol X^*_j-\pi(\boldsymbol X))-\sum_{j=1}^i\boldsymbol\eta_j\right\|=o\big((d^{8}\log d) n^{\lambda}\big),\quad \text{a.s.}.
\end{align}
This rate is same as \citet[Theorem 1.2]{gouezel2010almost} which is under the mixing condition tied to spectral properties. As $\lambda_k\asymp e^{-|k|^\gamma}$ with $\gamma>0$, the rate on the right hand side of \eqref{e:CorollaryPoly} can be improved to $o\big(n^{\frac13+\frac{2}{3(3p-2)}}\big)$ by our method. One may intuitively see this by letting $\delta_1$ and $\delta_2$ tend to $\infty$ in Theorem \ref{thm-asip}. 
\end{remark}

Let us briefly describe our approach as follows.  We divide the positive integer number $\N$ into the intervals $[2^m+1, 2^{m+1}]$ for $m=0,1,...$. Each interval is partitioned into a sequence of big blocks with length $2^{\alpha_1m}$ for $\alpha_1\in(0,1)$ and small blocks with length $m$. Following the properties of $\beta-$mixing, we can construct i.i.d. random variables distributed as the big block sums in $\HH$. { The projection of these random variables on $\HH_{\le d}$ are comparable with Gaussian random variables following \citet{gotze2011estimates}, while the residual on $\HH_{>d}$ and small block sums are negligible.} After carefully choosing the relation between the dimension $d$ and the sample size $n$, we can get the result.


\subsection{Organization of the paper and some notations}
The paper is organized as follows. Our main result is stated in Section \ref{sec:intro}. In Section \ref{sec:lemma}, we provide some preliminary lemmas and the technique for the proof of our main result. The proof of the ASIP is given in Section \ref{sec:main}. In Section \ref{sec:example}, we give some applications. The proof of the crucial lemmas in Section \ref{sec:lemma} is deferred to Appendix A and B,  and the proofs of examples in Section \ref{sec:example} are given in Appendix C.

We finish this section by introducing some notations which will be frequently used in sequel. For variables $(X_i)_{i\in\N_0}\in\HH$, $X_{i,j}$ denotes the $j-$th element of $X_i$.  We denote $\xi^*$ the random variable has the same distribution with $\xi$.  For an operator $A:\HH\to\HH$, $\|A\|_F$ for operator on $\HH$ denote the Frobenius norm of $A$, i.e., $\|A\|_F=\sqrt{\text{Tr}(AA^T)}$. For $x\in\R$, $\lceil x\rceil$ and $\lfloor x\rfloor$   smallest integer greater than $x$ and the integer part of $x$ respectively.
The symbols $C$ and $c$ denote positive numbers, $C_p$ and $c_p$ denote positive numbers depend on the parameter $p$. Their values may vary from line to line.

For the stationary process $(\boldsymbol X_k)_{k\in\N_0}$  with ergodic measure $\pi$, without loss of generality, from now on we assume 
$$\pi(\boldsymbol X_0)=\boldsymbol 0.$$

\section{Auxiliary Lemmas for Theorem \ref{thm-asip}}\label{sec:lemma}
The strategy of proving Theorem \ref{thm-asip} is to decompose $\sum_{i=1}^n \boldsymbol X_i$ into two parts by the block technique, showing that the big block sum is comparable with Gaussian random variables and the small block sum is negligible following the properties of $\beta-$mixing which are shown below. In this section, we provide some preliminary lemmas and the technique for the proof of our main result.
\subsection{Lemmas of $\beta-$mixing time series}
In this subsection, we recall the result of \cite{Berbee1987} first for the convenience of the reader, which is useful to construct the independence for $\beta-$mixing time series. At the end of this subsection, we give the following two lemmas, the first lemma paving a way for proving the last which is an extension to infinite dimensional space valued random variables of \citet[Theorem 4.1]{Shao96}.

\begin{lemma}\cite[Lemma 2]{Merle97}\label{lem:Merl}
Let  $\boldsymbol X$ and $\boldsymbol Y$ be two random variables defined on $\HH$ with quantile functions $Q_{\|\boldsymbol X\|}(u)$ and $Q_{\|\boldsymbol Y\|}(u)$. Then
\begin{eqnarray*}
|\E\Ll \boldsymbol X,\boldsymbol Y\Rr-\Ll\E \boldsymbol X, \E \boldsymbol Y\Rr|\le 18\int_0^{\bar\alpha}Q_{\|\boldsymbol X\|}(u)Q_{\|\boldsymbol Y\|}(u)\dif  u.
\end{eqnarray*}
Here $Q_{\|\boldsymbol X\|}(u)=\inf\{t:\PP(\|\boldsymbol X\|>t)\le u\}$ and $Q_{\|\boldsymbol Y\|}(u)$ is defined similarly. $\bar\alpha$ is defined by $\bar\alpha=\sup_{A\in\sigma(\boldsymbol X),B\in\sigma(\boldsymbol Y)}\big\{|\PP(A\cap B)-\PP(A)\PP(B)|\big\}.$ Similarly,
\begin{eqnarray*}
	{ \|\cov(\boldsymbol X,\boldsymbol Y)\|_F\le 18\int_0^{\bar\alpha}Q_{\|\boldsymbol X\|}(u)Q_{\|\boldsymbol Y\|}(u)\dif  u.}
\end{eqnarray*}
\end{lemma}
It is easy to see that $Q_{\| \boldsymbol X \|}$ is the inverse function of $\PP(\| \boldsymbol X \|>t)$ and is a non-increasing function. Moreover, $Q_{\|\boldsymbol  X \|}(U)\overset{\mathscr{D}}=\| \boldsymbol X \|$, where $U$ is a random variable uniformly distributed on $[0,1]$.

\begin{lemma}\label{lem:shao}
Suppose Assumption  {\bf (A1)} holds and $\boldsymbol X_0\sim\pi$ with $p-$th moment, there exists a constant $\bar C>0$, such that for any $2\le p'<p$ and large enough $n$,
\begin{eqnarray*}
\E\big\|\sum_{i=1}^n \boldsymbol X_i\big\|^{p'}\le \bar C n^{\frac {p'}2}
\big(\pi(\|\boldsymbol X\|^p)\big)^{\frac {p'}p}.
\end{eqnarray*}
\end{lemma}
\begin{proof}
The proof is given in Appendix B.
\end{proof}
\subsection{Blocking} \label{ss:blocking}
The blocking technique is a typical way for dependent time series and applied in proving almost sure invariance principle, see e.g., \cite{Liu09,gouezel2010almost,Berkes14}.

We subdivide $\mathbb{N}$ into the intervals $[2^m+1,2^{m+1}]$. For any positive integer $m$, let $m_1=2^{\lfloor\alpha_1m\rfloor}$ and $m_2=\lfloor C^*\log 2^m\rfloor$ with $0<\alpha_1<1$ and $C^*>0$, $\alpha_1$ and $C^*$ will be chosen later. For any $n\in[2^m+1,2^{m+1}]$, we denote
$$ \kappa(n)=\lfloor \frac{n-2^m}{m_1+m_2}\rfloor.$$
For $1\le j\le \kappa(2^{m+1})$, put
\begin{eqnarray*}
I_{m,j}&=&\{i: 2^m+(m_1+m_2)(j-1)+1\le i\le 2^m+(m_1+m_2)(j-1)+m_1\},\\
J_{m,j}&=&\{i: 2^m+(m_1+m_2)(j-1)+m_1+1\le i\le 2^m+(m_1+m_2)j\},
\end{eqnarray*}
where $I_{m,j}$(resp., $J_{m,j}$) are big (resp., small) blocks. The tail is defined by
\begin{eqnarray*}
I_{m,\kappa(2^{m+1})+1}&=&\{i: 2^m+(m_1+m_2)\kappa(2^{m+1})+1\le i\le 2^{m+1}\wedge\big((m_1+m_2)\kappa(2^{m+1})+m_1\big)\},\\
J_{m,\kappa(2^{m+1})+1}&=&\{i: 2^m+(m_1+m_2)\kappa(2^{m+1})+m_1+1\le i\le 2^{m+1}\}.
\end{eqnarray*}
Thus we decompose the interval $[2^m+1,2^{m+1}]$ as a union of $\kappa(2^{m+1})+1$ big blocks and $\kappa(2^{m+1})+1$ small blocks. We further denote $i_{m,j}$ is smallest element of $I_{m,j}$, $\mathcal{I}(m)=\bigcup_{j}I_{m,j}$ and $\mathcal{J}(m)=\bigcup_{j}J_{m,j}$. Let the block sums defined by
$$\boldsymbol Y_{m,j}=\sum_{i\in I_{m,j}} \boldsymbol X_i,~~\text{and}\quad \boldsymbol Z_{m,j}=\sum_{i\in J_{m,j}} \boldsymbol X_i.$$

We introduce following lemmas to give the moment bounds for $\boldsymbol Y_{m,j}$ and $\boldsymbol Z_{m,j}$ which paving a way for showing that $\boldsymbol Y_{m,j}$ is comparable with i.i.d. $\boldsymbol {\tilde Y_{m,j}}$ distributed as $\boldsymbol Y_{m,j}$ for $j=1,...,\kappa(n)$ and small blocks $\boldsymbol Z_{m,j}$ are negligible. Their proofs are given in Appendix C.

\begin{lemma}\label{lem:Y}
Under the condition of Theorem \ref{thm-asip}, for any $2\le {p'}<p$ and $j=1,...,\kappa(2^{m+1})+1$, we have
\begin{eqnarray*}
\E\|\boldsymbol Y_{m,j}\|^{p'}\le Cm_1^{\frac {p'}2},\quad\E\|\boldsymbol Z_{m,j}\|^{p'}\le Cm_2^{\frac {p'}2}.
\end{eqnarray*}
\end{lemma}
\begin{lemma}\label{lem:bigblock}
Under the condition of Theorem \ref{thm-asip}, for any $2<p'<p$, we can construct independent random variables $\boldsymbol {\tilde Y}_{m,j}$ defined on a richer probability space $(\Omega_1,\A_1,\PP_1)$ such that $\boldsymbol {\tilde Y}_{m,j}\overset{\mathscr{D}}=\boldsymbol Y_{m,j}$ for $j=1,...,\kappa(2^{m+1})$ satisfies
\begin{eqnarray}\label{e:bigblock1}
\max_{1\le i\le \kappa(2^{m+1})}\|\sum_{j=1}^{i}(\boldsymbol Y_{m,j}-\boldsymbol {\tilde Y}_{m,j})\|=o(2^{((1-\alpha_1)/{p'}+\alpha_1/2)m}\sqrt{\log 2^m}),\quad \text{a.s..}
\end{eqnarray}
Moreover,
\begin{eqnarray}\label{e:bigblock2}
	\max_{1\le i\le j}\|\sum_{\ell=1}^{i}(\boldsymbol Y_{m,\ell}-\boldsymbol {\tilde Y}_{m,\ell})\|=o(2^{((1-\alpha_1)/{p'}+\alpha_1/2 )m}\sqrt{\log 2^m}),\quad \text{a.s.}.
\end{eqnarray}

\end{lemma}

\begin{lemma}\label{lem:smallblock}
Under the condition of Theorem \ref{thm-asip}, for any $n\in[2^m+1,2^{m+1}]$, we have
\begin{eqnarray*}
\max_{2^m+1\le i\le n}\big\|\sum_{\ell\in\mathcal{J}(m)\cap[0,i]} \boldsymbol X_\ell\big\|=o(2^{\frac12(1-\alpha_1)m}\log 2^m),\quad\text{a.s.}.
\end{eqnarray*}
Moreover, for any  i.i.d.  centered Gaussian random vectors $\eta_\ell$,
\begin{eqnarray*}
	\max_{2^m+1\le i\le n}\big\|\sum_{\ell\in\mathcal{J}(m)\cap[0,i]}\boldsymbol \eta_\ell\big\|=o(2^{\frac12(1-\alpha_1)m}\log 2^m),\quad\text{a.s.}.
\end{eqnarray*}
\end{lemma}

\section{Proof of Theorem \ref{thm-asip}}\label{sec:main}
 
The strategy of proving Theorem \ref{thm-asip} is to decompose $\sum_{i=1}^n\boldsymbol X_i$ into two parts by the block technique where the small blocks are negligible following   Lemma \ref{lem:smallblock}. For big blocks, the projection of these random variables on $\HH_{\le d}$ are comparable with Gaussian random variables following \citet{gotze2011estimates}, while the residual on $\HH_{>d}$  can be directly estimated. The dimension $d$ will be chosen carefully with respect to $n$. In this section we first introduce following lemmas to show big blocks are comparable with Gaussian random variables and then finish the proof of Theorem \ref{thm-asip}.

\begin{lemma}\label{lem:Gamma}
	Under the conditions of Theorem \ref{thm-asip}, one has
	\begin{eqnarray}\label{e:Gammarate}
		\|\cov(\sum_{i=1}^{n}\boldsymbol  X_i)-n\boldsymbol \Gamma\|_F\le C,
	\end{eqnarray}
	where $C$ is a positive constant depends on $p$ and $\beta$. For the  diagonal elements of the matrix $\cov(\sum_{i=1}^{n} \boldsymbol X_i)$, we further have
	\begin{eqnarray}\label{e:dim}
		\E\big[\sum_{i=1}^n   X_{i,\ell} \big]^2-n\lambda_\ell= r_\ell,
	\end{eqnarray}
	where $X_{i,\ell}$ is the $\ell$-th element of $\boldsymbol X_i$ and $r_\ell$ is a square summable constant depends on $\ell$.
	
\end{lemma}
\begin{proof}
	A straight calculation yields that
	\begin{eqnarray*}
		\cov(\sum_{i=1}^{n} \boldsymbol X_i)=\sum_{i=1}^n \cov(\boldsymbol X_i)+\sum_{1\le i<j\le n} \big(\cov (\boldsymbol  X_i,\boldsymbol  X_j)+\cov (\boldsymbol  X_i, \boldsymbol X_j)^T\big).
	\end{eqnarray*}
	Since $(\boldsymbol  X_i)_{0\le i\le n}$ is stationary, there exist linear operators $\boldsymbol \Gamma_i$ for $i=0,1,2,...$ such that
	\begin{eqnarray*}
		\cov(\boldsymbol  X_i)=\boldsymbol \Gamma_0 \quad \cov (\boldsymbol  X_i,\boldsymbol  X_j)=\E[ \boldsymbol X_i  \boldsymbol X_j^T]=\E[ \boldsymbol X_0  \boldsymbol X_{j-i}^T]=\boldsymbol \Gamma_{j-i}.
	\end{eqnarray*}
	We further denote $\boldsymbol \Gamma=\boldsymbol \Gamma_0+\sum_{k=1}^{\infty}\big(\boldsymbol \Gamma_k+\boldsymbol \Gamma_k^T\big)$. Thus,
	\begin{eqnarray*}
		&&\big\|\cov(\sum_{i=1}^{n}\boldsymbol  X_i)-n\boldsymbol \Gamma\big\|_F\\
		&=&\big\| \sum_{i=1}^n \cov(\boldsymbol X_i)+\sum_{1\le i<j\le n} \big(\cov (\boldsymbol X_i,\boldsymbol  X_j)+\cov (\boldsymbol  X_i,\boldsymbol  X_j)^T\big)- n\boldsymbol \Gamma_0-n\sum_{k=1}^{\infty}\big(\boldsymbol \Gamma_k+\boldsymbol \Gamma_k^T\big)  \big\|_F\\
		&=&\big\|\sum_{k=1}^n (n-k)(\boldsymbol \Gamma_k+\boldsymbol \Gamma_k^T)-n\sum_{k=1}^{\infty}\big(\boldsymbol \Gamma_k+\boldsymbol \Gamma_k^T \big)    \big\|_F\\
		&\le&\big\|n\sum_{k=n+1}^\infty (\boldsymbol \Gamma_k+\boldsymbol \Gamma_k^T)\big\|+\big\|\sum_{k=1}^{n}k\big(\boldsymbol \Gamma_k+\boldsymbol \Gamma_k^T \big)    \big\|_F.
	\end{eqnarray*}
	To finish the proof, we show that $\|\cov (\boldsymbol X_0, \boldsymbol X_k)\|_F=\|\boldsymbol \Gamma_{k}\|_F\le C e^{-ck}$ and this gives the bound of \eqref{e:Gammarate}. Following Lemma \ref{lem:Merl} and H\"{o}lder's inequality, one has
	\begin{eqnarray*}
		\|\cov ( \boldsymbol X_0,  \boldsymbol X_k)\|_F&\le& 18 \int_0^{\alpha(k)}Q_{\|  \boldsymbol X_0\|}^2(u)\dif u\\
		&\le& 18\big(\int_0^{1}1_{\{u\le\alpha(k)\}}\dif u\big)^{1-2/p} (\int_0^1 Q_{\|  \boldsymbol X_0\|}^p(u)\dif u)^{2/p} \\
		&\le& Ce^{-k\beta (1-2/p)} (\E\| \boldsymbol X\|^p)^{2/p}.
	\end{eqnarray*}
Thus, we obtain
\begin{eqnarray*}
\big\|\cov(\sum_{i=1}^{n}  \boldsymbol X_i)-n\boldsymbol \Gamma\big\|_F\le  C(\E\|  \boldsymbol X\|^p)^{2/p}.
\end{eqnarray*}

Notice that $\boldsymbol\Gamma$ is diagonal with element $\lambda_1,\lambda_2,...$ and with the form $\boldsymbol\Gamma=\boldsymbol\Gamma_0+\sum_{k=1}^\infty(\boldsymbol\Gamma_k+\boldsymbol\Gamma_k^T)$, thus for each element of $\boldsymbol\Gamma$, we have for the $\ell-$th element in the diagonal
\begin{eqnarray}
\E [  X_{0,\ell}^2]+ 2\sum_{k=1}^\infty\E[ X_{0,\ell}  X_{k,\ell}]=\lambda_\ell, \label{e:diag}
\end{eqnarray}
and for the remaining elements with $i\neq j$ and $i,j=1,2,...$,
\begin{eqnarray}
\E[  X_{0,i} X_{0,j}]+\sum_{k=1}^{\infty}\big(\E[  X_{0,i}  X_{k,j}]+\E[  X_{0,j}  X_{k,i}]\big)  =0.   \label{e:exceptdiag}
\end{eqnarray}
Following \eqref{e:Gammarate} and the definition of Frobenius norm, it is easy to see that
\begin{eqnarray*}
	\sum_{\ell=1}^{\infty} \Big(\E\big[\sum_{i=1}^n   X_{i,\ell}\big]^2-n\lambda_\ell\Big)^2 \le C,
\end{eqnarray*}
which implies there exists a square summable number $r_\ell$ such that
\begin{eqnarray*}
	\E\big[\sum_{i=1}^n   X_{i,\ell}\big]^2-n\lambda_\ell= r_\ell.
\end{eqnarray*}
\end{proof}
For any vectors $\boldsymbol \xi_0,...,\boldsymbol \xi_n$ defined on $\HH$, let
$$
\boldsymbol \xi_i^{(d)}=(\xi_{i,1},...,\xi_{i,d})^T\in\R^d,
\quad
\boldsymbol \xi_i^{[d]}=(0,...,0,\xi_{i,d+1},\xi_{i,d+2},...)^T\in \HH.
$$

\begin{lemma}\label{lem:eigenvalue}
Under the conditions of Theorem \ref{thm-asip} and $m_1\lambda_k>Ck+c$ for $k\in\N$ and constants $C$ and $c$, let $\lambda_{\max}$ and $\lambda_{\min}$ be the largest and smallest eigenvalue of $\cov(\sum_{i=1}^{m_1}\boldsymbol X_i^{(d)})$ respectively, then one has
one has
\begin{eqnarray*}
	\lambda_{\max}\le C m_1\lambda_1\quad\lambda_{\min} \ge c{m_1}\lambda_d.
\end{eqnarray*}
\end{lemma}	
\begin{proof}
Since $ \boldsymbol X_i=((\boldsymbol X_i^{(d)})^T,0,...)^T+ \boldsymbol X_i^{[d]}$, a straight calculation implies that
$\cov(\sum_{i=1}^{m_1} \boldsymbol X_i^{(d)})$ is the upper left $d\times d$ block of of $\cov(\sum_{i=1}^{m_1}  \boldsymbol X_i)$. Following \citet[Corollary II.A.3]{garren1968bounds}, we have
\begin{eqnarray*}
\lambda_{\max}&\le& \max_{k\in\{1,...d\}}\sum_{j=1}^d\big|\E\big[\sum_{i=1}^{m_1} X_{i,k}\sum_{i=1}^{m_1}X_{i,j}\big]\big|\\
\lambda_{\min}&\ge& \min_{k\in\{1,...d\}}\Big(\big|\E\big[\sum_{i=1}^{m_1} X_{i,k}\big]^2\big|-\sum_{j=1,j\neq k}^d\big|\E\big[\sum_{i=1}^{m_1} X_{i,k}\sum_{i=1}^{m_1} X_{i,j}\big]\big|\Big).
\end{eqnarray*}
Since
\begin{eqnarray*}
&&\sum_{j=1,j\neq k}^d\big|\E\big[\sum_{i=1}^{m_1} X_{i,k}\sum_{i=1}^{m_1} X_{i,j}\big]\big|\\
&=&\sum_{j=1,j\neq k}^d\big|\sum_{i=1}^{{m_1}} \E[X_{i,k} X_{i,j}]+\E[\sum_{1\le i<r\le {m_1}} X_{i,k} X_{r,j}+ X_{i,j} X_{r,k}] \big|\\
&=&\sum_{j=1,j\neq k}^d \big|{m_1}\E[ X_{0,k} X_{0,j}]+\sum_{i=1}^{m_1}({m_1}-i)\E[ X_{0,k} X_{i,j}+X_{0,j} X_{i,k}]\big|,
\end{eqnarray*}
\eqref{e:exceptdiag} and Lemma \ref{lem:Merl} yield
\begin{eqnarray*}
	&&\sum_{j=1,j\neq k}^d\big|\E\big[\sum_{i=1}^{m_1} X_{i,k}\sum_{i=1}^{m_1} X_{i,j}\big]\big|\\
	&=&\sum_{j=1,j\neq k}^d \big|\sum_{i={m_1}+1}^\infty {m_1}\E[ X_{0,k} X_{i,j}+ X_{0,j} X_{i,k}]+\sum_{i=1}^{m_1} i\E[ X_{0,k} X_{i,j}+ X_{0,j} X_{i,k}]\big|\\
	&\le& C(d-1).	
\end{eqnarray*}
Combining with \eqref{e:dim} and the fact $(r_k)_{k\ge0}^2$ is a sequence of summable constants, we obtain
\begin{eqnarray*}
	\lambda_{\max}&\le& \max_{k\in\{1,...d\}}\sum_{j=1}^d\big|\E\big[\sum_{i=1}^{m_1} X_{i,k}\sum_{i=1}^{m_1} X_{i,j}\big]\big|\\
	&\le& \max_{k\in\{1,...d\}}\big(C(d-1)+m_1\lambda_k+r_k\big)\le Cd+m_1\lambda_1+c\\
	\lambda_{\min}&\ge& \min_{k\in\{1,...d\}}\Big(\big|\E\big[\sum_{i=1}^{m_1} X_{i,k}\big]^2\big|-\sum_{j=1,j\neq k}^d\big|\E\big[\sum_{i=1}^{m_1} X_{i,k}\sum_{i=1}^{m_1} X_{i,j}\big]\big|\Big)\\
	&\ge& \min_{k\in\{1,...d\}}({m_1}\lambda_k+r_k-C(d-1)) \ge  m_1\lambda_d-Cd+c.
\end{eqnarray*}
Since ${m_1}\lambda_1\ge {m_1}\lambda_d>Cd+c$, thus
$$\lambda_{\max}\le Cm_1\lambda_1,\quad \lambda_{\min}\ge cm_1\lambda_d.
$$
\end{proof}

\begin{lemma}\label{lem:Zaitsev}
	Under the conditions of Theorem \ref{thm-asip} and  ${m_1}\lambda_k>Ck+c$ for $k\in\N$. For any $2\le p'<p$ and $\boldsymbol{\tilde Y}_{m,j}$, one can construct on a probability space $(\Omega_2,\A_2,\PP_2)$ a sequence of independent random vectors  $(\boldsymbol{\tilde Y}_{m,j})^{*}$ and the corresponding sequence of independent Gaussian random vectors $$\boldsymbol \eta_{m,j}=((\boldsymbol \eta_{m,j}^{(d)})^T,0,...)+\boldsymbol \eta_{m,j}^{[d]},$$ 
	for $1\le j\le \kappa(2^{m+1})$ so that 
	$$ \boldsymbol{\tilde Y}_{m,j}\overset{\mathscr{D}}{=}(\boldsymbol{\tilde Y}_{m,j})^*,\ \ \boldsymbol \eta_{m,j}^{[d]}\sim \mathcal N\big(\boldsymbol 0,\cov (\boldsymbol{\tilde Y}_{m,j}^{[d]})\big),\ \ \boldsymbol \eta_{m,j}^{(d)}\sim \mathcal N\big(\boldsymbol 0,\cov (\boldsymbol{\tilde Y}_{m,j}^{(d)})\big)$$
	 and
	\begin{eqnarray}\label{e:main7}
		&&\max_{1\le i\le \kappa(2^{m+1})} \|\sum_{j=1}^i(\boldsymbol{\tilde Y}_{m,j})^*-\boldsymbol \eta_{m,j}\|\\
		&=&o\Big(A_d^{\frac1{p'}} \lambda_{d}^{-\frac {1}2} 2^{((1-\alpha_1)/p'+\alpha_1/2)m} \sqrt{\log 2^{m}}+2^{m/2}d^{(1-\delta_2)/2}\log 2^m \Big)\quad \text{a.s.} .\nonumber
		\end{eqnarray}	
	 
\end{lemma}
\begin{proof}
It is easy to see that
$$
\boldsymbol{\tilde Y}_{m,j}=((\boldsymbol{\tilde Y}_{m,j}^{(d)})^T,0,0,...)^T+\boldsymbol{\tilde Y}_{m,j}^{[d]}.
$$
To compare $(\boldsymbol{\tilde Y}_{m,j})_{1\le j\le\kappa(2^{m+1})}$ with Gaussian random vectors on Hilbert space $\HH$, we show that $\boldsymbol{\tilde Y}_{m,j}^{(d)}$ are comparable with Gaussian random variables $\boldsymbol \eta_{m,j}^{(d)}$ on $\R^d$ space following  \citet[Theorem 2]{gotze2011estimates} and $\boldsymbol{\tilde Y}_{m,j}^{[d]}$ are negligible.

{\textit{Step $1$.}} We show that for $\boldsymbol{\tilde Y}_{m,j}^{(d)}$ one can construct on a probability space $(\Omega_2,\mathcal A_2,\PP_2)$ a sequence of independent random vectors $(\boldsymbol{\tilde Y}_{m,j}^{(d)})^*\overset{\mathscr{D}}{=}\boldsymbol{\tilde Y}_{m,j}^{(d)}$ and the corresponding sequence of independent Gaussian random vectors $\boldsymbol \eta_{m,j}^{(d)}\sim \mathcal N(\boldsymbol{0}, \cov(\boldsymbol{\tilde Y}_{m,j}^{(d)}))$ for $1\le j\le \kappa(2^{m+1})$ such that 
\begin{equation}\label{e:step1}
\max_{1\le i\le \kappa(2^{m+1})} \|\sum_{j=1}^i\big((\boldsymbol{\tilde Y}_{m,j}^{(d)})^*-\boldsymbol \eta_{m,j}^{(d)}\big)\| =
o\Big(A_d^{\frac1{p'}} \lambda_{d}^{-\frac {1}2} 2^{((1-\alpha_1)/p'+\alpha_1/2)m} \sqrt{\log 2^{m}}\Big),\quad \text{a.s.},
\end{equation}
where $A_d=C \max\{ d^{11p'},d^{\frac{p'(p'+2)}4}(\log d)^{\frac{p'(p'+1)}2}\}$ and $C$ is a constant depends on  $p'$.

{ For brevity, instead of writing out the properties of $(\boldsymbol{\tilde Y}_{m,j}^{(d)})^*$ and $\boldsymbol \eta_{m,j}^{(d)}$ listed above we simply say that \textit{there is a construction of $\boldsymbol{\tilde Y}_{m,j}^{(d)}$} to show that one can construct a coupling between $\boldsymbol{\tilde Y}_{m,j}^{(d)}$ and Gaussian random variables on a probability space enjoying the mentioned additional properties accordingly for $1\le j\le\kappa(2^{m+1})$.   
}

Let $\lambda_{\max,Y}$ and $\lambda_{\min,Y}$ be the maximal and minimal strictly positive eigenvalues of the covariance matrix $\cov(\boldsymbol{\tilde Y}_{m,j}^{(d)})$ respectively. According to \citet[Theorem 2]{gotze2011estimates}, there is a construction of $\boldsymbol{\tilde Y}_{m,j}^{(d)}$ on probability space $(\Omega_2,\mathcal A_2,\PP_2)$  such that
\begin{eqnarray*}
	\E\big[\max_{1\le i\le \kappa(2^{m+1})} \|\sum_{j=1}^i	\big((\boldsymbol{\tilde Y}_{m,j}^{(d)})^*-\boldsymbol \eta_{m,j}^{(d)}\big)\|^{p'}\big]
	&\le& A_d (\lambda_{\max,Y}/\lambda_{\min,Y})^{\frac {p'}2} \kappa(2^{m+1}) \E\|\boldsymbol{\tilde Y}_{m,j}\|^{p'}\\
	&\le& C A_d (\lambda_{\max,Y}/\lambda_{\min,Y})^{\frac {p'}2} \kappa(2^{m+1}) m_1^{\frac {p'}2}.
\end{eqnarray*}
Following Lemma \ref{lem:eigenvalue}, we obtain
\begin{eqnarray*}
	\E\big[\max_{1\le i\le \kappa(2^{m+1})} \|\sum_{j=1}^i\big((\boldsymbol{\tilde Y}_{m,j}^{(d)})^*-\boldsymbol \eta_{m,j}^{(d)}\big)\|^{p'}\big]
	&\le& C A_d \lambda_{d}^{-\frac {p'}2}\kappa(2^{m+1}) m_1^{\frac {p'}2}.
\end{eqnarray*}
Thus, we have
\begin{eqnarray*}
	&&\sum_{m=1}^{\infty}\PP\big(\max_{1\le i\le \kappa(2^{m+1})} \|\sum_{j=1}^i	\big((\boldsymbol{\tilde Y}_{m,j}^{(d)})^*-\boldsymbol \eta_{m,j}^{(d)}\big)\| \ge
	A_d^{\frac1{p'}} \lambda_{d}^{-\frac {1}2} 2^{((1-\alpha_1)/p'+\alpha_1/2)m} \sqrt{\log 2^{m}}
	\big)\\
	&\le& \sum_{m=1}^{\infty} \E\big[\max_{1\le i\le \kappa(2^{m+1})} \|\sum_{j=1}^i	\big((\boldsymbol{\tilde Y}_{m,j}^{(d)})^*-\boldsymbol \eta_{m,j}^{(d)}\big)\|^{p'}\big] A_d^{-1} \lambda_{d}^{\frac {p'}2} 2^{-((1-\alpha_1)+\alpha_1 p'/2)m} (\log 2^{m})^{-\frac{p'}{2}}\\
	&\le&  \sum_{m=1}^{\infty} C(\log 2^{m})^{-\frac{p'}{2}}<\infty.
\end{eqnarray*}
By the Borel-Cantelli lemma, we obtain
\begin{eqnarray*}
	\max_{1\le i\le \kappa(2^{m+1})} \|\sum_{j=1}^i	\big((\boldsymbol{\tilde Y}_{m,j}^{(d)})^*-\boldsymbol \eta_{m,j}^{(d)}\big)\| =
	o\Big(A_d^{\frac1{p'}} \lambda_{d}^{-\frac {1}2} 2^{((1-\alpha_1)/p'+\alpha_1/2)m} \sqrt{\log 2^{m}}\Big),\quad \text{a.s.}.
\end{eqnarray*}

\textit{Step $2$.} We show that there is a construction of $\boldsymbol{\tilde Y}_{m,j}^{[d]}$ on probability space $(\Omega_2,\mathcal A_2,\PP_2)$ such that  
\begin{equation}\label{e:step2}
\max_{1\le i\le \kappa(2^{m+1})} \|\sum_{j=1}^i\big((\boldsymbol{\tilde Y}_{m,j}^{[d]})^*- \boldsymbol \eta_{m,j}^{[d]}\big)\|
=o(2^{(1-\alpha_1)m/2}(1+2^{\alpha_1m}d^{1-\delta_2})^{1/2}\log 2^m)\quad \text{a.s.} .
\end{equation}

According to \citet[(25,28,29)]{gotze2011estimates}, one has
\begin{eqnarray*}
	&&\E\big[\max_{1\le i\le \kappa(2^{m+1})} \|\sum_{j=1}^i\big((\boldsymbol{\tilde Y}_{m,j}^{[d]})^*-\boldsymbol \eta_{m,j}^{[d]}\big)\|^{2}\big]\\
	&\le& C  \E\big[\max_{1\le i\le \kappa(2^{m+1})} \|\sum_{j=1}^i(\boldsymbol{\tilde Y}_{m,j}^{[d]})^*\|^{2}\big]
	+C \E\big[\max_{1\le i\le \kappa(2^{m+1})} \|\sum_{j=1}^i\boldsymbol \eta_{m,j}^{[d]}\|^{2}\big]	\\
	&\le& C\sum_{j=1}^{\kappa(2^{m+1})}\E\|\boldsymbol{\tilde Y}_{m,j}^{[d]}\|^2.
\end{eqnarray*}
By \eqref{e:dim} and the condition ${m_1}\lambda_k>Ck+c\ge r_k$ for $k\in\N$, we can get
\begin{eqnarray*}
	\E\big[\max_{1\le i\le \kappa(2^{m+1})} \|\sum_{j=1}^i\big((\boldsymbol{\tilde Y}_{m,j}^{[d]})^*-\boldsymbol \eta_{m,j}^{[d]}\big)\|^{2}\big] 
	&\le& C \kappa(2^{m+1})\sum_{k=d+1}^{\infty} \E[\sum_{i=1}^{m_1} X_{i,k}]^2\\
	&=&  C \kappa(2^{m+1})\sum_{k=d+1}^{\infty}(m_1\lambda_k+r_k)\\
	&\le& C 2^{m}d^{1-\delta_2}.
\end{eqnarray*}	
Thus, the Markov inequality yields
\begin{eqnarray*}
	&&\sum_{m=1}^{\infty}\PP\Big( \max_{1\le i\le \kappa(2^{m+1})} \|\sum_{j=1}^i\big((\boldsymbol{\tilde Y}_{m,j}^{[d]})^*-\boldsymbol \eta_{m,j}^{[d]}\big)\|\ge  2^{m/2}d^{(1-\delta_2)/2}\log 2^m \Big)\\
	&\le&\sum_{m=1}^{\infty} \E\big[\max_{1\le i\le \kappa(2^{m+1})} \|\sum_{j=1}^i\big((\boldsymbol{\tilde Y}_{m,j}^{[d]})^*-\boldsymbol \eta_{m,j}^{[d]}\big)\|^{2}\big]
	2^{-m}d^{-1+\delta_2}\log^{(-2)} 2^m\\
	&<&\infty.
\end{eqnarray*}
By the Borel-Cantelli lemma, we obtain
\begin{eqnarray*}
	\max_{1\le i\le \kappa(2^{m+1})} \|\sum_{j=1}^i\big((\boldsymbol{\tilde Y}_{m,j}^{[d]})^*-\boldsymbol \eta_{m,j}^{[d]}\big)\|
	=
	o\Big(2^{m/2}d^{(1-\delta_2)/2}\log 2^m\Big),\quad \text{a.s.}.
\end{eqnarray*}

\textit{Step $3$.} Combining the estimates in step $1$ and step $2$ above, i.e., \eqref{e:step1} and \eqref{e:step2}, one has
\begin{eqnarray*}
&&\max_{1\le i\le \kappa(2^{m+1})} \|\sum_{j=1}^i\big((\boldsymbol{\tilde Y}_{m,j})^*-\boldsymbol \eta_{m,j}\big)\|\\
&\le& \max_{1\le i\le \kappa(2^{m+1})} \|\sum_{j=1}^i\big((\boldsymbol{\tilde Y}_{m,j}^{(d)})^*-\boldsymbol \eta_{m,j}^{(d)}\big)\|+\max_{1\le i\le \kappa(2^{m+1})} \|\sum_{j=1}^i\big((\boldsymbol{\tilde Y}_{m,j}^{[d]})^*-\boldsymbol \eta_{m,j}^{[d]}\big)\|\nonumber\\
&=&o\Big(A_d^{\frac1{p'}} \lambda_{d}^{-\frac {1}2} 2^{((1-\alpha_1)/p'+\alpha_1/2)m} \sqrt{\log 2^{m}}+2^{m/2}d^{(1-\delta_2)/2}\log 2^m \Big)\quad \text{a.s.} .\nonumber
\end{eqnarray*}	
\end{proof}

\begin{proof}[Proof of Theorem \ref{thm-asip}]
For any $i\in[2^m+1,2^{m+1}]$, it is easy to see that
\begin{eqnarray*}
	\sum_{\ell=1}^i   \boldsymbol X_\ell=\sum_{j=1}^{\kappa(i)} \boldsymbol Y_{m,j}+\sum_{\ell\in I_{m,\kappa(i)+1}\cap[2^m+1,i]}  \boldsymbol X_\ell+\sum_{\ell\in\mathcal{J}(m)\cap [2^m+1,i]}  \boldsymbol X_\ell.
\end{eqnarray*}
Recall that $i_{m,j}$ is the smallest element of $I_{m,j}$, following Lemma \ref{lem:smallblock}, we can get
\begin{eqnarray}\label{e:main1}
	&&\max_{2^{m}+1\le i\le 2^{m+1}}\|\sum_{\ell=1}^i   \boldsymbol X_\ell-\sum_{j=1}^{\kappa(i)}\boldsymbol  Y_{m,j}\|\nonumber\\
	&\le&\max_{2^{m}+1\le i\le 2^{m+1}}\|\sum_{\ell\in I_{m,\kappa(i)+1}\cap[2^m+1,i]} \boldsymbol  X_\ell\|+	\max_{2^{m}+1\le i\le 2^{m+1}}\|\sum_{\ell\in\mathcal{J}(m)\cap [2^m+1,i]}\boldsymbol  X_\ell\|\nonumber\\
	&=&\max_{1\le j\le \kappa(2^{m+1})}\max_{1\le i<|I_{m,j}|}\big\|\sum_{\ell=i_{m,j}}^{i_{m,j}+i}  \boldsymbol X_\ell\big\|+o(2^{\frac12(1-\alpha_1)m}\log 2^m),\quad{\mathrm{a.s..}}
\end{eqnarray}
Let $p'$ be a positive constant such that $2< p'<p$. For the first term, \citet[Proposition 1]{Wu07} and Lemma \ref{lem:Y} imply
\begin{eqnarray*}
	\Big(\E[\max_{1\le i\le 2^r }\|\sum_{j=1}^i   \boldsymbol X_j\|^{p'}]\Big)^{\frac1{p'}}\le\sum_{i=0}^r2^{(r-i)/p'}\Big(\E\|\sum_{j=1}^{2^i}\boldsymbol  X_j\|^{p'}\Big)^{\frac1{p'}}
	\le C\sum_{i=0}^r2^{(r-i)/p'}(2^i)^{\frac12}\le C2^{\frac r2},
\end{eqnarray*}
which yields
\begin{eqnarray*}
	\E\big[\max_{1\le i< |I_{m,j}|}\|\sum_{\ell=i_{m,j}}^{i_{m,j}+i}  \boldsymbol X_\ell\|^{p'}\big]\le |I_{m,j}|^{\frac {p'}2}\le 2^{\frac {p'}2 \alpha_1m}.
\end{eqnarray*}
Thus, the Markov inequality implies
\begin{eqnarray*}
	&&\sum_{m=1}^\infty\PP\big(\max_{1\le j\le \kappa(2^{m+1})}\max_{1\le i< |I_{m,j}|}\|\sum_{\ell=i_{m,j}}^{i_{m,j}+i} \boldsymbol  X_\ell\|\ge 2^{((1-\alpha_1)/p'+\alpha_1/2)m}(\log 2^m)^{1/2}\big)\\
&\le& \sum_{m=1}^\infty\E\big[\max_{1\le j\le \kappa(2^{m+1})}\max_{1\le i< |I_{m,j}|}\|\sum_{\ell=i_{m,j}}^{i_{m,j}+i}  \boldsymbol X_\ell\|^{p'}\big]2^{-(1-\alpha_1+\alpha_1 p'/2)m}(\log 2^m)^{-p'/2}\\
&\le& \sum_{m=1}^\infty\sum_{j=1}^{\kappa(2^{m+1})}\E\big[\max_{1\le i< |I_{m,j}|}\|\sum_{\ell=i_{m,j}}^{i_{m,j}+i}  \boldsymbol X_\ell\|^{p'}\big]2^{-(1-\alpha_1+\alpha_1 p'/2)m}(\log 2^m)^{-p'/2}\\
&\le& \sum_{m=1}^\infty  2^{(1-\alpha_1+p'\alpha_1/2)m}  2^{-(1-\alpha_1+p'\alpha_1/2)m}(\log 2^m)^{-p'/2} <\infty.
\end{eqnarray*}
By the Borel-Cantelli lemma, we obtain
\begin{eqnarray}\label{e:main2}
\max_{1\le j\le\kappa(2^{m+1})}\max_{i< |I_{m,j}|}\|\sum_{\ell=i_{m,j}}^{i_{m,j}+i} \boldsymbol X_\ell\|=o(2^{((1-\alpha_1)/p'+\alpha_1/2)m}(\log 2^m)^{1/2}),\quad \text{a.s.}.
\end{eqnarray}
Similar estimate holds for any  i.i.d.  centered Gaussian random vectors $\boldsymbol \eta_\ell$, i.e.,
\begin{eqnarray}\label{e:main6}
\max_{1\le j\le\kappa(2^{m+1})}\max_{i< |I_{m,j}|}\|\sum_{\ell=i_{m,j}}^{i_{m,j}+i}\boldsymbol \eta_\ell\|=o(2^{((1-\alpha_1)/p'+\alpha_1/2)m}(\log 2^m)^{1/2}),\quad \text{a.s.}.
\end{eqnarray}
Combining \eqref{e:main1} and \eqref{e:main2}, we have
\begin{align*}
\max_{2^{m}+1\le i\le 2^{m+1}}\|\sum_{\ell=1}^i \boldsymbol X_\ell-\sum_{j=1}^{\kappa(i)} \boldsymbol Y_{m,j}\|
=o\big((2^{\frac12(1-\alpha_1)m}+2^{((1-\alpha_1)/p'+\alpha_1/2)m})\log 2^m\big),\quad\text{a.s.}.
\end{align*}
Following Lemma \ref{lem:bigblock}, on a richer probability space of $(\Omega,\A,\PP)$, one has
\begin{eqnarray}\label{e:main3}
&& \max_{2^{m}+1\le i\le 2^{m+1}}\|\sum_{\ell=1}^i  \boldsymbol X_\ell-\sum_{j=1}^{\kappa(i)} \boldsymbol{\tilde Y}_{m,j}\|\nonumber\\
&=&\max_{2^{m}+1\le i\le 2^{m+1}}\|\sum_{\ell=1}^i   \boldsymbol X_\ell-\sum_{j=1}^{\kappa(i)}\boldsymbol Y_{m,j}+\sum_{j=1}^{\kappa(i)}\boldsymbol Y_{m,j}-\sum_{j=1}^{\kappa(i)}\boldsymbol{\tilde Y}_{m,j}\|\nonumber\\
&=&o\big((2^{\frac12(1-\alpha_1)m}+2^{((1-\alpha_1)/p'+\alpha_1/2)m})\log 2^m\big),\quad\text{a.s.}.
\end{eqnarray}

Lemma \ref{lem:Zaitsev} implies that there is a construction of $\boldsymbol{\tilde Y}_{m,j}$ on probability space $(\Omega_2,\A_2,\PP_2)$ and one can compare $\boldsymbol{\tilde Y}_{m,j}$ with i.i.d. Gaussian random variables $\boldsymbol \eta_{m,j}$. Now we regularize $\boldsymbol \eta_{m,j}$, that is, replacing the covariance matrix $\cov(\boldsymbol \eta_{m,j})$  by the linear form $|\boldsymbol I_{m,j}|\boldsymbol\Gamma$. We denote 
$\boldsymbol{\tilde\eta}_{m,j}:=\mathcal N(\boldsymbol 0,|\boldsymbol I_{m,j}|\boldsymbol\Gamma)$.

Following Lemma \ref{lem:Gamma} and for large enough $m$, one has
\begin{eqnarray}\label{e:normalsum}
	\cov(\boldsymbol \eta_{m,j})+\boldsymbol M_{m,j}=|\boldsymbol I_{m,j}|\boldsymbol\Gamma+\boldsymbol N_{m,j},
\end{eqnarray}
where $\boldsymbol M_{m,j}$ and $\boldsymbol N_{m,j}$ are  positive definite linear operators with $\|\boldsymbol M_{m,j}\|_F \le C$, $\|\boldsymbol N_{m,j}\|_F\le C$. Therefor, $\mathcal N(\boldsymbol 0,|\boldsymbol I_{m,j}|\boldsymbol\Gamma+\boldsymbol N_{m,j})$ is the sum of  $\boldsymbol{\tilde\eta}_{m,j}$ and an independent random variable $\mathcal N(\boldsymbol 0,\boldsymbol N_{m,j})$.
On the other hand, \eqref{e:normalsum} implies $\mathcal N(\boldsymbol 0,|\boldsymbol I_{m,j}|\boldsymbol\Gamma+\boldsymbol N_{m,j})$ is also the sum of  $\boldsymbol \eta_{m,j}$ and an independent random variable $\mathcal N(\boldsymbol 0,\boldsymbol M_{m,j})$. Using \citet[Lemma A1]{berkes1979approximation}, we obtain a coupling between $\boldsymbol \eta_{m,j}$ and $\boldsymbol {\tilde\eta}_{m,j}$ such that the difference $\boldsymbol D_{m,j}= \boldsymbol \eta_{m,j}-\boldsymbol{\tilde\eta}_{m,j}$ is centered and $\E\|\boldsymbol D_{m,j}\|^2\le C$.

Thus, following the L\'{e}vy inequality, see \citet[5.4.a]{lin2011probability}, we have
\begin{eqnarray*}
	&&\PP\big(\max_{2^{m}+1\le i\le 2^{m+1}}\|\sum_{j=1}^{\kappa(i)} \boldsymbol D_{m,j}\|> 2^{\frac12(1-\alpha_1)m}{\log 2^m}\big)\\
	&\le& 2\PP\big(\|\sum_{j=1}^{\kappa(2^{m+1})} \boldsymbol D_{m,j}\|> 2^{\frac12(1-\alpha_1)m}{\log 2^m}\big)\\
	&\le& 2\E\|\sum_{j=1}^{\kappa(2^{m+1})} \boldsymbol D_{m,j}\|^2  2^{-(1-\alpha_1 )m}(\log 2^m)^{-2}\\
	&\le& C\kappa(2^{m+1}) 2^{-(1-\alpha_1 )m}(\log 2^m)^{-2},
\end{eqnarray*}
which is summable with respect to $m$. Then there is a construction for $\boldsymbol \eta_{m,j}$ on probability space $(\Omega_3,\A_3,\PP_3)$ such that
\begin{eqnarray}\label{e:main4}
\max_{2^{m}+1\le i\le 2^{m+1}}\|\sum_{j=1}^{\kappa(i)}(\boldsymbol \eta_{m,j}^*-\boldsymbol {\tilde\eta}_{m,j}^*)\|=o(2^{\frac12(1-\alpha_1)m}\log 2^m ),\quad\text{a.s..}
\end{eqnarray}
Combining \eqref{e:main4} with \eqref{e:main7} and \citet[Lemma 4.1]{Berkes14}, there is a construction for $\boldsymbol{\tilde Y}_{m,j}^*$ on probability space $(\Omega_4,\A_4,\PP_4)$ such that
\begin{align}\label{e:main5}
&\max_{2^{m}+1\le i\le 2^{m+1}}\|\sum_{j=1}^{\kappa(i)} (\boldsymbol {\tilde  Y}_{m,j}^{**}-\boldsymbol{\tilde \eta}_{m,j}^{**})\|\\
=&o\Big(\big(A_d^{\frac1{p'}} \lambda_{d}^{-\frac {1}2} 2^{(\frac{1-\alpha_1}{p'}+\frac{\alpha_1}2)m} +2^{m/2}d^{(1-\delta_2)/2}+2^{(1-\alpha_1)m/2}\big) \log 2^m
\Big)\quad \text{a.s.} .\nonumber
\end{align}
Using \citet[Lemma 4.1]{Berkes14} again with \eqref{e:main6}, \eqref{e:main3}, \eqref{e:main5} and Lemma \ref{lem:smallblock}, we can finally construct a probability space $(\Omega_{(m)},\A_{(m)},\PP_{(m)})$ on which we can define $\boldsymbol X^*_\ell$ distributed as $\boldsymbol X_\ell$ and i.i.d. Gaussian random variables $\boldsymbol \eta_\ell\sim \mathcal N(\boldsymbol 0,\boldsymbol\Gamma)$ such that
\begin{eqnarray*}
&&\max_{2^{m}+1\le i\le 2^{m+1}}\|\sum_{\ell=1}^i   \boldsymbol X_\ell^*-\sum_{\ell=1}^i \boldsymbol \eta_\ell\|\\
&=&
o\Big(\big(A_d^{\frac1{p'}} \lambda_{d}^{-\frac {1}2} 2^{(\frac{1-\alpha_1}{p'}+\frac{\alpha_1}2)m}  +2^{m/2}d^{(1-\delta_2)/2}+2^{(1-\alpha_1)m/2}
\big)\log 2^m \Big)\quad \text{a.s.} .\nonumber
\end{eqnarray*}
Recall the definition of $A_d$ and condition in Lemma \ref{lem:Zaitsev},
$$A_d^{\frac1{p'}}=C \max\{ d^{11},d^{\frac{p'+2}4}(\log d)^{\frac{p'+1}2}\},\quad
m_1\lambda_k> Ck+c,$$
and $C_1d^{-\delta_1}\le \lambda_d\le C_2d^{-\delta_2}$, $d=2^{m\theta_{p'}}$. We take $\alpha_1=(1+\delta_1)\theta_{p'}$ and
\begin{eqnarray*}
\theta_{p'}=\min\big\{\frac{p'-2}{22p'-2+(2p'-2)\delta_1+p'\delta_2},\frac{p'-2}{p'(p'+2)/2-2+(2p'-2)\delta_1+p'\delta_2}\big\}
\end{eqnarray*}
to ensure
\begin{eqnarray*}
A_d^{\frac1{p'}} d^{\frac {\delta_1}2} 2^{(\frac{1-\alpha_1}{p'}+\frac{\alpha_1}2)m} =2^{\frac{1-\alpha_1}2m}(1+2^{\alpha_1m}d^{1-\delta_2})^{1/2}.
\end{eqnarray*}
Then we have
\begin{eqnarray*}
	&&\max_{2^m+1\le i\le 2^{m+1}}\|\sum_{\ell=2^m+1}^{i}( \boldsymbol X^*_\ell-\boldsymbol \eta_\ell)\|=	o\big(2^{\frac12m(1-(\delta_2-1)\theta_{p'})}(\log 2^m)^{\frac{p'+3}2}\big),\quad \text{a.s..}
\end{eqnarray*}
Hence we can construct probability space $(\Omega',\A',\PP')$ on which
\begin{eqnarray*}
	\max_{1\le i\le 2^{m+1}}\|\sum_{\ell=1}^{i}( \boldsymbol X^*_\ell-\boldsymbol \eta_\ell)\|&=&	o\big(2^{\frac12m(1-(\delta_2-1)\theta_{p'})}(\log 2^m)^{\frac{p'+3}2}\big)\quad \text{a.s.}.
\end{eqnarray*}
{
For the logarithmic term, there exists a $2<p''<p'$ such that,
\begin{eqnarray*}
	\max_{1\le i\le 2^{m+1}}\|\sum_{\ell=1}^{i}( \boldsymbol X^*_\ell-\boldsymbol \eta_\ell)\|&=&	o\big(2^{\frac12m(1-(\delta_2-1)\theta_{p''})}\big)\quad \text{a.s.}.
\end{eqnarray*}
Let $p''$ close enough to $p$, we obtain 
\begin{eqnarray*}
	\max_{1\le i\le 2^{m+1}}\|\sum_{\ell=1}^{i}( \boldsymbol X^*_\ell-\boldsymbol \eta_\ell)\|&=&	o(2^{m\bar\theta})\quad \text{a.s.},
\end{eqnarray*}
where
\begin{align}\label{e:theta}
	\bar\theta>\max\big\{\frac{(2p-2)\delta_1+2\delta_2+23p-4}{44p-4+(4p-4)\delta_1+2p\delta_2},\frac{(2p-2)\delta_1+2\delta_2+p(p+4)/2-4}{p(p+2)-4+(4p-4)\delta_1+2p\delta_2}\big\},
\end{align}
and consequently
\begin{eqnarray*}
	&&\max_{1\le i\le n}\|\sum_{\ell=1}^{i}( \boldsymbol X^*_\ell-\boldsymbol \eta_\ell)\|=	o(n^{\bar\theta}),\quad \text{a.s..}
\end{eqnarray*}
}
\end{proof}

\begin{proof}[Proof of Corollary \ref{coro-asip}]
When $\delta_1=\delta_2=\delta$, it is easy to see that	
\begin{align}\label{e:theta}
\bar\theta>\max\big\{\frac{2p\delta+23p-4}{44p-4+(6p-4)\delta},\frac{2p \delta +p(p+4)/2-4}{p(p+2)-4+(6p-4)\delta }\big\},
\end{align}	
which converges to $\frac13+\frac{2}{3(3p-2)}$ as $\delta\to\infty$. That is, for any $\e>0$, as $\delta>\bar\delta_{p,\e}$ where
\begin{eqnarray}\label{e:delta}
\bar\delta_{p,\e}=\left\{
\begin{aligned}
	\frac{25p^2-54p+8-(44p-4)(3p-2)\e}{2(3p-2)^2\e},\quad p\le 42,\\
	\frac{p^3/2+3p^2-12p'+8-(p^2+2p-4)(3p-2)\e}{2(3p-2)^2\e}, \quad p> 42.
\end{aligned}
\right.
\end{eqnarray}
we obtain
\begin{align*}
\max_{1\le i\le n }\big\|\sum_{j=1}^i ( \boldsymbol X^*_j-\boldsymbol \eta_j)\big\|=o\big(n^{\frac13+\frac{2}{3(3p-2)}+\e}\big),\quad \text{a.s.}.
\end{align*}
\end{proof}


\section{Examples}\label{sec:example}
 {    In this section, we give two examples where the first compares the mixing condition with the geometric moment contraction (GMC) condition of \citet[Example 2.2]{Berkes14} on $\HH$ and the second considers the functional autoregressive processes.
}
\subsection{Markov chain}\label{ex:1}
Let $(\boldsymbol X_k)_{k \ge 0}$ be a  $\HH-$valued time homogeneous Markov chain with $p-$th moment for $p>2$ satisfying the condition:

{\bf(A3)} $(\boldsymbol X_k)_{k\ge0}$ is irreducible, aperiodic and Feller. There exists a Lyapunov function $V:{ \HH}\to [1,+\infty)$ such that
	\begin{align}\label{e:lya}
	\E[V(\boldsymbol X_{1})|\boldsymbol X_0=\boldsymbol x]\le\gamma V(\boldsymbol x)+K1_\textbf{C}(\boldsymbol x),
	\end{align}
	where $0<\gamma<1$, $K>0$ and \textbf{C} is a compact set.

It is easy to prove that $(\boldsymbol X_k)_{k\ge0}$ is exponential ergodic with invariant measure $\pi$ which yields exponential $\beta-$mixing, see \citet[Theorem 2.1]{Tweedie94} and  \citet[Proposition 1]{Davy1973}. We further assume that $\boldsymbol X_0\sim \pi$ and { $\boldsymbol\Gamma$ is a positive definite operator whose eigenvalues polynomial decay}, then $(\boldsymbol X_k)_{k\ge0}$ satisfies ASIP with rate \eqref{e:remark}. 


Comparing {\bf(A3)} with the contraction condition in \citet{Berkes14},
\begin{equation}\label{e:GMC}
\E\|\boldsymbol X_1^{\boldsymbol x}-\boldsymbol X_1^{\boldsymbol y}\|^p\le r^p\|\boldsymbol x-\boldsymbol y\|^p,
\end{equation} 
where $0<r<1$, $\boldsymbol X_1^{\boldsymbol x}$ denotes the Markov chain with initial value $\boldsymbol X_0=\boldsymbol x$. The term $K 1_\textbf{C}(\boldsymbol x)$ of {\bf(A3)} makes it be a weaker condition than the contraction condition.

\subsection{Functional autoregressive processes}\label{ex:3}
We consider the functional autoregressive processes
\begin{eqnarray}\label{e:AR1}
\boldsymbol X_{k+1}=\boldsymbol \mu+\boldsymbol A\boldsymbol X_{k}+\boldsymbol B \boldsymbol\e_{k+1},
\end{eqnarray}
where $\boldsymbol A$ and $\boldsymbol B$ are linear operators from $\mathbb H$ to $\mathbb H$ with kernel $  \boldsymbol a(\cdot,\cdot)$ and $  \boldsymbol b(\cdot,\cdot)$ respectively, $\boldsymbol \e_{k+1}$ is white noise. Let $(\boldsymbol e_k)_{k\in\N}$ be an orthonormal basis of $L^2([0,\pi])$ with the form $\boldsymbol e_k(x)=\sqrt{\frac{2}{\pi}}\sin (kx)$ and 
$$\boldsymbol A\boldsymbol e_k(x)=\lambda_k \boldsymbol e_k(x),
$$
where $\boldsymbol A\boldsymbol f(x)=\int_0^\pi \boldsymbol a(s,x)\boldsymbol f(s)\dif s$. According to Karhunen-Lo\`eve decomposition, one has
$$
 \boldsymbol a(s,t)=\sum_{k=1}^\infty \lambda_k  \boldsymbol e_k(s) \boldsymbol e_k(t).
$$
\eqref{e:AR1} can be written as
\begin{eqnarray}\label{e:AR1(1)}
 \boldsymbol X_{k+1}(\cdot)= \boldsymbol\mu(\cdot)+\int_{0}^{\pi} \boldsymbol a(\cdot,s) \boldsymbol X_{k}(s)\dif s+\int_{0}^{\pi} \boldsymbol b(\cdot,s)  \boldsymbol \e_{k+1}(s)\dif s.
\end{eqnarray}
We refer the reader to \citet{bosq2000linear,wang2020functional} for more details of functional autoregressive processes. We assume $ \boldsymbol\mu= \boldsymbol0$,  $ \boldsymbol B^2= \boldsymbol A$ and $ \boldsymbol A$ is symmetric for the simplify of calculation. Further assuming that { $0<\lambda_k \asymp k^{-\delta}<1$}, conditions {\bf (A1)} and {\bf (A2)} are satisfied and $( \boldsymbol X_k)_{k\in \N_0}$ satisfies the ASIP.

\begin{appendix}
\section{The proof of Lemma \ref{lem:shao} }
The proof of Lemma \ref{lem:shao} following the properties of $\alpha-$mixing sequence, see \cite{Bradley2005} for more details. Denote the $\alpha-$mixing coefficients by
\begin{eqnarray*}
\alpha(n)=\sup_{k\ge1}\big\{|\PP(A\cap B)-\PP(A)\PP(B)|:A\in\sigma(X_i,1\le i\le k),B\in\sigma(X_i,i\ge k+n)\big\}.
\end{eqnarray*}
Since $\alpha(n)\le\beta(n)$, assumption {\bf (A1)} implies
\begin{eqnarray}\label{e:alpha}
	\alpha(n)\le C e^{-\beta n}.
\end{eqnarray}
We first give  following preparing lemma.
\begin{lemma}\label{lem:shaoBDG}
Let $( \boldsymbol\theta_i)_{1\le i\le n}$ be a sequence of random variables on $\HH$ with finite $p-$moment and let $\F_i=\sigma(\boldsymbol\theta_j, j\le i)$. Then for any $p\ge2$, there exists constant $C_p$ such that
\begin{eqnarray*}
\E\big\|\sum_{i=1}^n \boldsymbol\theta_i\big\|^p&\le& C_p\Big(\big(\sum_{i=1}^n\E\|\boldsymbol\theta_i\|^2\big)^{\frac p2}+\sum_{i=1}^n\E\|\boldsymbol\theta _i\|^p+n^{p-1}\sum_{i=1}^n\E\|\E[\boldsymbol\theta_i|\mathcal{F}_{i-1}]\|^p\\
&&+n^{\frac p2-1}\sum_{i=1}^n\E\big|\E[\|\boldsymbol\theta_i\|^2|\mathcal{F}_{i-1}]-\E\|\boldsymbol\theta_i\|^2\big|^{\frac p2}
\Big).
\end{eqnarray*}
\end{lemma}
\begin{proof}
The strategy is to construct martingale differences $\boldsymbol\theta_i-\E[\boldsymbol\theta_i|\mathcal{F}_{i-1}]$ and using the Burkholder inequality to get the result. That is,
\begin{eqnarray*}
\E\big\|\sum_{i=1}^n \boldsymbol\theta_i\big\|^p &=& \E\big\|\sum_{i=1}^n(\boldsymbol\theta_i-\E[\boldsymbol\theta_i|\mathcal{F}_{i-1}]+\E[\theta_i|\mathcal{F}_{i-1}])\big\|^p\\
&\le& 2^p\Big(\E\big\|\sum_{i=1}^n(\boldsymbol\theta_i-\E[\boldsymbol\theta_i|\mathcal{F}_{i-1}])\big\|^p+n^{p-1}\sum_{i=1}^n\E\|\E[\boldsymbol\theta_i|\mathcal{F}_{i-1}]\|^p\Big).
\end{eqnarray*}
For the first term, \citet[Theorem 4.1]{Pinelis94} implies
\begin{eqnarray*}
&&\E\big\|\sum_{i=1}^n(\boldsymbol\theta_i-\E[\boldsymbol\theta_i|\mathcal{F}_{i-1}])\big\|^p\\
&\le&
C_p \Big(\sum_{i=1}^n\E\big\|\boldsymbol\theta_i-\E[\boldsymbol\theta_i|\mathcal{F}_{i-1}]\big\|^p+\E\Big[\sum_{i=1}^n\E\big[\|\boldsymbol\theta_i-\E[\boldsymbol\theta\boldsymbol\theta_i|\mathcal{F}_{i-1}]\|^2|\mathcal{F}_{i-1}\big]\Big]^{\frac p2}\Big)\\
&\le&C_p \Big(\sum_{i=1}^n\E\big\|\boldsymbol\theta_i-\E[\boldsymbol\theta_i|\mathcal{F}_{i-1}]\big\|^p+
\E\Big[\sum_{i=1}^n\E[\|\boldsymbol\theta_i\|^2|\mathcal{F}_{i-1}]\Big]^{\frac p2}\Big)\\
&=&C_p \Big(\sum_{i=1}^n\E\big\|\boldsymbol\theta_i-\E[\boldsymbol\theta_i|\mathcal{F}_{i-1}]\big\|^p+
\E\Big[\sum_{i=1}^n2\E[\|\boldsymbol\theta_i\|^2|\mathcal{F}_{i-1}]-\E\|\boldsymbol\theta_i\|^2+\E\|\boldsymbol\theta_i\|^2\Big]^{\frac p2}\Big)\\
&\le&  C_p\Big(\sum_{i=1}^n\E\|\boldsymbol\theta_i\|^p+ \big(\sum_{i=1}^n\E\|\boldsymbol\theta_i\|^2\big)^{\frac p2}+
\E\big(\sum_{i=1}^n(\E[\|\boldsymbol\theta_i\|^2|\mathcal{F}_{i-1}]-\E\|\boldsymbol\theta_i\|^2)\big)^{\frac p2}\Big)\\
&\le& C_p\Big(\sum_{i=1}^n\E\|\boldsymbol\theta_i\|^p+ \big(\sum_{i=1}^n\E\|\boldsymbol\theta_i\|^2\big)^{\frac p2}+
n^{\frac p2-1}\sum_{i=1}^n\E|\E[\|\boldsymbol\theta_i\|^2|\mathcal{F}_{i-1}]-\E\|\boldsymbol\theta_i\|^2|^{\frac p2}\Big).
\end{eqnarray*}
Thus we can get the result.
\end{proof}

\begin{proof}[Proof of Lemma \ref{lem:shao}]
We first prove the case ${p'}=2$ which is a extension of \cite{Rio93} on $\HH$. The fact that $(\boldsymbol X_i)_{1\le i\le n}$ is zero mean implies
\begin{eqnarray*}
\E\Big\|\sum_{i=1}^n \boldsymbol X_i\Big\|^2&=&\sum_{1\le i\le n}\sum_{1\le j\le n}\E\Ll \boldsymbol  X_i,  \boldsymbol X_j\Rr\\
&=&\sum_{1\le i\le n}\E\| \boldsymbol  X_i\|^2+2\sum_{1\le i<j\le n}\big(\E\Ll  \boldsymbol X_i, \boldsymbol  X_j\Rr-\Ll\E  \boldsymbol X_i,\E  \boldsymbol X_j\Rr\big)\\
&\le&\sum_{1\le i\le n}\E\|\boldsymbol  X_i\|^2+2\sum_{1\le i<j\le n} 18\int_0^{\alpha(j-i)}Q_{\| \boldsymbol X_i\|}^2(u)\dif u,
\end{eqnarray*}
where the last line follows Lemma \ref{lem:Merl} and $\alpha(j-i)\ge\bar\alpha$ therein.
For the second term, a straight calculation yields
\begin{eqnarray}\label{e:Merl1}
&&\sum_{1\le i< j\le n}\int_0^{\alpha(j-i)}Q_{\| \boldsymbol X_i\|}^2(u)\dif u \\
&=&\sum_{1\le i<n}\int_0^{1} \sum_{i< j\le n}1_{\{\alpha(j-i)>u\}}Q_{\| \boldsymbol X_i\|}^2(u) \dif u
\nonumber\\
&\le& \sum_{1\le i\le n}\Big(\int_0^1 \big(\sum_{i< j\le n}1_{\{\alpha(j-i)>u\}}\big)^{\frac{p}{p-2}}\dif u\Big)^{\frac{p-2}{p}}
\Big(\int_0^1  Q_{\| \boldsymbol X_i\|}^p(u)\dif u\Big)^{\frac{2}{p}}.\nonumber
\end{eqnarray}
Since $\sum_{1\le j\le n}1_{\{\alpha(j)>u\}}=k$ if and only if $\alpha(k+1)\le u<\alpha(k)$, then we have
\begin{eqnarray*}
\int_0^1 \big(\sum_{i< j\le n}1_{\{\alpha(j-i)>u\}}\big)^{\frac{p}{p-2}}\dif u
&\le&\int_0^1 \big(\sum_{j=1}^n 1_{\{\alpha(j)>u\}}\big)^{\frac{p}{p-2}}\dif u\\
&=&\sum_{k=1}^\infty \int^{\alpha(k)}_{\alpha(k+1)}\big(\sum_{j=1}^n1_{\{\alpha(j)>u\}}\big)^{\frac{p}{p-2}}\dif u.\\
&\le&\sum_{k=1}^\infty k^{\frac{p}{p-2}} \alpha(k)<\infty.
\end{eqnarray*}
Notice that $\int_0^1Q_{\| \boldsymbol X_i\|}^p(u)\dif u=\E\| \boldsymbol X_i\|^{p}$, combining estimates above with \eqref{e:Merl1}, we obtain
\begin{eqnarray}\label{e:Merl}
2\sum_{1\le i<j\le n} 18\int_0^{\alpha(j-i)}Q_{\|\boldsymbol  X_i\|}^2(u)\dif u\le \sum_{1\le i\le n}C(\E\|\boldsymbol  X_i\|^p)^{\frac{2}{p}}
\end{eqnarray}
Thus, the stationary of $(\boldsymbol X_i)_{0\le i\le n}$ implies
\begin{eqnarray}\label{e:Rio}
\E\Big\|\sum_{i=1}^n \boldsymbol X_i\Big\|^2\le C n \big(\pi(\| \boldsymbol X\|^p)\big)^{\frac{2}{p}},
\end{eqnarray}

We shall prove the case $p>p'>2$ by induction on $n$. Suppose that for $1\le k<n$,
\begin{eqnarray}\label{e:induc1}
\E\Big\|\sum_{i=1}^k \boldsymbol  X_i\Big\|^{p'} \le  Ck^{\frac {p'}2}\big(\pi(\| \boldsymbol X\|^p)\big)^{\frac {p'}p}.
\end{eqnarray}
When  $k=n$, let $m=\lfloor \sqrt n\rfloor$ and $\bar\kappa(n)=\lfloor \frac{n}{2m}\rfloor$ here. The block sums are defined by
$$\boldsymbol Y_{i,1}=\sum_{j=1+2(i-1)m}^{n\wedge(2i-1)m} \boldsymbol X_j;\quad
\boldsymbol Y_{i,2}=\sum_{j=1+(2i-1)m}^{n\wedge 2im} \boldsymbol X_j\quad i\in\{1,...,\bar\kappa(n)+1\}.
$$
Then we can get
\begin{eqnarray}\label{e:shao}
\E\Big\|\sum_{i=1}^n \boldsymbol X_i\Big\|^{p'}
&=&\E\Big\|\sum_{j=1}^{\bar\kappa(n)+1}\boldsymbol Y_{j,1}+\sum_{j=1}^{\bar\kappa(n)+1}\boldsymbol Y_{j,2}\Big\|^{p'}
\le 2^{{p'}-1}\Big(\E\big\|\sum_{j=1}^{\bar\kappa(n)+1}\boldsymbol Y_{j,1}\big\|^{p'}+\E\big\|\sum_{j=1}^{\bar\kappa(n)+1}\boldsymbol Y_{j,2}\big\|^{p'}\Big)\nonumber\\
&:=&2^{{p'}-1}(I_1+I_2).
\end{eqnarray}
For $I_1$, Lemma \ref{lem:shaoBDG} implies
\begin{eqnarray*}
I_1&\le& C_{p'}\Big(\big(\sum_{i=1}^{\bar\kappa(n)+1}\E\|\boldsymbol Y_{i,1}\|^2\big)^{\frac {p'}2}+(\bar\kappa(n)+1)^{p'-1}\sum_{i=1}^{\bar\kappa(n)+1}\E\|\E[\boldsymbol Y_{i,1}|\mathcal{F}_{i-1}]\|^{p'}\\
&&+(\bar\kappa(n)+1)^{\frac {p'}2-1}\sum_{i=1}^{\bar\kappa(n)+1}\E\big|\E[\|\boldsymbol Y_{i,1}\|^2|\mathcal{F}_{i-1}]-\E\|\boldsymbol Y_{i,1}\|^2\big|^{\frac {p'}2} +\sum_{i=1}^{\bar\kappa(n)+1}\E\|\boldsymbol Y_{i,1}\|^{p'}\Big)\\
&:=&C_{p'}\Big(I_{1,1}+I_{1,2}+I_{1,3}+I_{1,4}\Big),
\end{eqnarray*}
where $\F_i=\sigma(\boldsymbol Y_{j,1},j\le i)$. For $I_{1,1}$, \eqref{e:Rio} implies
\begin{eqnarray*}
I_{1,1}
\le \big(\sum_{i=1}^{\bar\kappa(n)+1}Cm\big(\pi(\| \boldsymbol X\|^p)\big)^{\frac 2p}\big)^{\frac {p'}2}\le (Cm(\bar \kappa(n)+1))^{\frac {p'}2} \big(\pi(\| \boldsymbol X\|^p)\big)^{\frac {p'}p}
\end{eqnarray*}
For $I_{1,2}$, since $\E[\boldsymbol Y_{i,1}|\F_{i-1}]$ is $\F_{i-1}$ measurable, one has
\begin{eqnarray*}
I_{1,2}&=&(\bar\kappa(n)+1)^{p'-1}\sum_{i=1}^{\bar\kappa(n)+1}\E \Ll \E[\boldsymbol Y_{i,1}|\mathcal{F}_{i-1}], \E[\boldsymbol Y_{i,1}|\mathcal{F}_{i-1}]\|\E[\boldsymbol Y_{i,1}|\mathcal{F}_{i-1}]\|^{p'-2}\Rr\\
&=&(\bar\kappa(n)+1)^{p'-1}\sum_{i=1}^{\bar\kappa(n)+1}\E \Ll \boldsymbol Y_{i,1}, \E[\boldsymbol Y_{i,1}|\mathcal{F}_{i-1}]\|\E[\boldsymbol Y_{i,1}|\mathcal{F}_{i-1}]\|^{p'-2}\Rr\\
&=&(\bar\kappa(n)+1)^{p'-1}\sum_{i=1}^{\bar\kappa(n)+1}\sum_{j=1+2(i-1)m}^{n\wedge(2i-1)m}\E \Ll \boldsymbol   X_j, \E[\boldsymbol Y_{i,1}|\mathcal{F}_{i-1}]\|\E[\boldsymbol Y_{i,1}|\mathcal{F}_{i-1}]\|^{p'-2}\Rr.
\end{eqnarray*}
Since $( \boldsymbol X_i)_{i\ge0}$  are  zero mean, Lemma \ref{lem:Merl} and Young's inequality yield
\begin{eqnarray*}
I_{1,2}&\le& 18(\bar\kappa(n)+1)^{p'-1}\sum_{i=1}^{\bar\kappa(n)+1}\sum_{j=1+2(i-1)m}^{n\wedge(2i-1)m}\int_0^{\alpha(m)}Q_{\| \boldsymbol X_j\|}(u)Q_{\|\E[\boldsymbol Y_{i,1}|\mathcal{F}_{i-1}]\|^{p'-1}}(u)\dif u\\
&\le& \frac{18}{p'}(\bar\kappa(n)+1)^{p-1}\sum_{i=1}^{\bar\kappa(n)+1}\sum_{j=1+2(i-1)m}^{n\wedge(2i-1)m}\int_0^{\alpha(m)}Q_{\| \boldsymbol X_j\|}^{p'}(u) \dif u\\
&&+ \frac{18(p'-1)}{p'}(\bar\kappa(n)+1)^{p'-1}m\sum_{i=1}^{\bar\kappa(n)+1}\int_0^{\alpha(m)}Q_{\|\E[\boldsymbol Y_{i,1}|\mathcal{F}_{i-1}]\|^{p'-1}}^{\frac{p'}{p'-1}}(u)\dif u,
\end{eqnarray*}
For the first term, H\"{o}lder's inequality implies
\begin{eqnarray*}
\int_0^{\alpha(m)}Q_{\| \boldsymbol X_j\|}^{p'}(u) \dif u &\le& \big(\int_0^11_{\{\alpha(m)>u\}}\dif u\big)^{1-\frac {p'}p} \big(\int_0^1Q_{\| \boldsymbol X_j\|}^p(u)\big)^{\frac {p'}p}\\
&=& \alpha(m)^{1-\frac {p'}p}(\E\| \boldsymbol X_j\|^p)^{\frac {p'}p}.
\end{eqnarray*}
For the second term, similar calculation implies
\begin{eqnarray*}
\int_0^{\alpha(m)}Q_{\|\E[\boldsymbol Y_{i,1}|\mathcal{F}_{i-1}]\|^{p'-1}}^{\frac {p'}{p'-1}}(u)\dif u
&\le&
\big(\int_0^1 1_{\{\alpha(m)>u\}}\dif u\big)^{1-\frac {p'}p}
\big(\int_0^1 Q_{\|\E[\boldsymbol Y_{i,1}|\mathcal{F}_{i-1}]\|^{p'-1}}^{\frac {p}{p'-1}}(u)\dif u\big)^{\frac {p'}p}\\
&=& \alpha(m)^{1-\frac {p'}p}\big(\E[\|\E[\boldsymbol Y_{i,1}|\F_{i-1}]\|^{(p'-1)\frac{p}{p'-1}}]\big)^{\frac {p'}p}\\
&\le& \alpha(m)^{1-\frac {p'}p}(\E\|\boldsymbol Y_{i,1}\|^{p})^{\frac {p'}p}\\
&\le&\alpha(m)^{1-\frac {p'}p} m^{p'} \big(\pi(\| \boldsymbol X\|^p)\big)^{\frac {p'}p}.
\end{eqnarray*}
Then we have,
\begin{eqnarray*}
I_{1,2} &\le& \frac{18}{p'}(\bar\kappa(n)+1)^{p'-1}\sum_{i=1}^{\bar\kappa(n)+1}\sum_{j=1+2(i-1)m}^{n\wedge(2i-1)m}\alpha(m)^{1-\frac {p'}p}(\E\| \boldsymbol X_j\|^p)^{\frac {p'}p}\\
&&+ \frac{18(p'-1)}{p'}(\bar\kappa(n)+1)^{p'-1}m\sum_{i=1}^{\bar\kappa(n)+1}\alpha(m)^{1-\frac {p'}p} m^{p'} \big(\pi(\| \boldsymbol X\|^p)\big)^{\frac {p'}p}\\
&\le& \left(\frac{18}{p'}(\bar \kappa(n)+1)^{p'}m+\frac{18(p'-1)}{p'}(\bar\kappa(n)+1)^{p'}m^{p'+1} \right)\alpha(m)^{1-\frac {p'}p}\big(\pi(\|\boldsymbol  X\|^p)\big)^{\frac {p'}p}\\
&\le& 36(\bar\kappa(n)+1)^{p'}m^{p'+1}\big(Ce^{-\beta m}\big)^{1-\frac {p'}p}\big(\pi(\| \boldsymbol X\|^p)\big)^{\frac {p'}p},
\end{eqnarray*}
where the last line follows \eqref{e:alpha}.
For $I_{1,3}$, we denote
$$\boldsymbol {\bar Y}_{i,1}=\E[\|\boldsymbol Y_{i,1}\|^2|\mathcal{F}_{i-1}]-\E\|\boldsymbol Y_{i,1}\|^2,$$
$\mathcal{F}_{i-1}-$measurable. Following the definition of $\boldsymbol Y_{i,1}$ and  \citet[Corollary]{Davy1968}, one has
\begin{eqnarray*}
\E|\boldsymbol {\bar Y}_{i,1}|^{\frac {p'}2}&=& \E\Big[|\boldsymbol {\bar Y}_{i,1}|^{\frac {p'}2-1}\mathrm{sgn}(\boldsymbol {\bar Y}_{i,1})\big(\E[\|\boldsymbol Y_{i,1}\|^2|\mathcal{F}_{i-1}]-\E\|\boldsymbol Y_{i,1}\|^2\big)\Big]\\
&=&\sum_{j,l=1+2(i-1)m}^{n\wedge(2i-1)m}\E\Big[|\boldsymbol {\bar Y}_{i,1}|^{\frac {p'}2-1}\mathrm{sgn}(\boldsymbol {\bar Y}_{i,1})\big(\E[\Ll \boldsymbol X_j,  \boldsymbol X_l\Rr|\mathcal{F}_{i-1}]-\E[\Ll \boldsymbol X_j, \boldsymbol X_l\Rr]\big)\Big]\\
&=&\sum_{j,l=1+2(i-1)m}^{n\wedge(2i-1)m}\Big\{\E\big[|\boldsymbol {\bar Y}_{i,1}|^{\frac {p'}2-1}\mathrm{sgn}(\boldsymbol {\bar Y}_{i,1})\Ll \boldsymbol X_j,  \boldsymbol X_l\Rr\big]-\E\big[|\boldsymbol {\bar Y}_{i,1}|^{\frac {p'}2-1}\mathrm{sgn}(\boldsymbol {\bar Y}_{i,1})\big]\E[\Ll \boldsymbol X_j,  \boldsymbol X_l\Rr]\Big\}\\
&\le& 12\sum_{j,l=1+2(i-1)m}^{n\wedge(2i-1)m} \alpha(m)^{1-\frac 2p-\frac{p'-2}{p'}}\big[\E|\Ll \boldsymbol X_j,  \boldsymbol X_l\Rr|^{\frac p2}\big]^{\frac 2p} \big[\E[|\boldsymbol {\bar Y}_{i,1}|^{\frac {p'}2-1}]^{\frac{p'}{p'-2}}\big]^{\frac{p'-2}{p'}}\\
&\le& 12m^2 \alpha(m)^{\frac 2{p'}-\frac 2p} \big(\pi(\| \boldsymbol X\|^p)\big)^{\frac 2p}\big[\E|\boldsymbol {\bar Y}_{i,1}|^{\frac {p'}2}\big]^{\frac{p'-2}{p'}}.
\end{eqnarray*}
Thus, we can get
\begin{eqnarray*}
\E|\boldsymbol {\bar Y}_{i,1}|^{\frac {p'}2}&\le& 12^{\frac {p'}2}m^{p'}\alpha(m)^{1-\frac {p'}p}\big(\pi(\| \boldsymbol X\|^p)\big)^{\frac {p'}p}.
\end{eqnarray*}
Then we have
\begin{eqnarray*}
I_{1,3}&\le&( 1+\bar \kappa(n))^{\frac {p'}2} 12^{\frac {p'}2}m^{p'}\alpha(m)^{1-\frac {p'}p}(\E\|\boldsymbol X\|^p)^{\frac {p'}p}\\
&\le& 12^{\frac {p'}2} (mn)^{\frac {p'}2} \big(Ce^{-\beta m}\big)^{1-\frac {p'}p}\big(\pi(\| \boldsymbol X\|^p)\big)^{\frac {p'}p}.
\end{eqnarray*}
For $I_{1,4}$, the induction hypothesis \eqref{e:induc1} implies
\begin{eqnarray*}
I_{1,4}&\le& \bar C(\bar \kappa(n)+1)m^{\frac {p'}2}\big(\pi(\| \boldsymbol X\|^p)\big)^{\frac {p'}p}\\
&\le& \bar C n^{\frac12+\frac {p'}4} \big(\pi(\| \boldsymbol X\|^p)\big)^{\frac {p'}p}.
\end{eqnarray*}
Combining the estimate of $I_{1,1}$, $I_{1,2}$, $I_{1,3}$ and $I_4$, one has
\begin{eqnarray*}
	I_{1}&\le& C_{p'}\Big((Cm( \bar\kappa(n)+1))^{\frac {p'}2}+ 36(\bar\kappa(n)+1)^{p'}m^{p'+1} \big(Ce^{-\beta m}\big)^{1-\frac {p'}p}\\
	&&\quad\quad+12^{\frac {p'}2}(mn)^{\frac {p'}2} \big(Ce^{-\beta m}\big)^{1-\frac {p'}p}+ \bar Cn^{\frac12+\frac {p'}4}\Big)\big(\pi(\|\boldsymbol  X\|^p)\big)^{\frac {p'}p}\\
	&\le& C_{p'}\Big(C_{1,p'}n^{\frac {p'}2}+ C_{2,p'}n^{\frac 34p'} e^{-\beta\sqrt n(1-\frac {p'}p)}+ \bar C n^{\frac12+\frac {p'}4}\Big)\big(\pi(\| \boldsymbol X\|^p)\big)^{\frac {p'}p}.
	\end{eqnarray*}
Similarly,
\begin{eqnarray*}
I_{2}&\le& C_{p'}\Big(C_{1,p'}n^{\frac {p'}2}+ C_{2,p'}n^{\frac 34p'} e^{-\beta\sqrt n(1-\frac {p'}p)}+ \bar C n^{\frac12+\frac {p'}4}\Big)\big(\pi(\| \boldsymbol X\|^p)\big)^{\frac {p'}p}.
\end{eqnarray*}
Combining the estimate of $I_1$, $I_2$ with \eqref{e:shao}, we have
\begin{eqnarray*}
\E\Big\|\sum_{i=1}^nX_i\Big\|^{p'}
&\le& 2^{p'}C_{p'}\Big(C_{1,p'}n^{\frac {p'}2}+ C_{2,p'}n^{\frac 34p'} e^{-\beta\sqrt n(1-\frac {p'}p)}+ \bar C n^{\frac12+\frac {p'}4}\Big)\big(\pi(\| \boldsymbol X\|^p)\big)^{\frac {p'}p}\\
&\le& \bar C n^{\frac {p'}2}\big(\pi(\| \boldsymbol X\|^p)\big)^{\frac {p'}p},
\end{eqnarray*}
here we take $\bar C/3\ge 2^{p'} C_{p'} C_{1,p'}$, $n$ is large enough.

\end{proof}

\section{The proof of Lemmas in Section \ref{ss:blocking}}
\begin{proof}[Proof of Lemma \ref{lem:Y}]
Following Lemma \ref{lem:shao}, we can get the result immediately.
\end{proof}

\begin{proof}[Proof of Lemma \ref{lem:bigblock}]

{ For $m$ large enough, \citet[Lemma 2.1]{Berbee1987} implies that one can construct independent random variables $\boldsymbol{\tilde Y}_{m,j}$ distributed as $\boldsymbol Y_{m,j}$ for $j=1,...,\kappa(2^{m+1})$ on a richer probability space and
\begin{eqnarray*}
\PP(\boldsymbol Y_{m,j}\neq\boldsymbol{\tilde Y}_{m,j}\text{~~for~some~}1\le j\le \kappa(2^{m+1}))&\le& \kappa(2^{m+1})\beta(m_2)\\
&\le& C \kappa(2^{m+1})  e^{-\beta m_2}.
\end{eqnarray*}}
The Markov inequality implies
\begin{eqnarray*}
&&\PP\Big(\max_{1\le i\le \kappa(2^{m+1})}\|\sum_{j=1}^{i}(\boldsymbol Y_{m,j}-\boldsymbol{\tilde Y}_{m,j})\|\ge x\Big)\\
&\le& \E\Big[\max_{1\le i\le \kappa(2^{m+1})}\|\sum_{j=1}^{i}(\boldsymbol Y_{m,j}-\boldsymbol{\tilde Y}_{m,j})\|^{p'}\Big]x^{-p'}\\
&=& x^{-p'}\E\Big[\max_{1\le i\le \kappa(2^{m+1})	}\|\sum_{j=1}^{i}(\boldsymbol Y_{m,j}-\boldsymbol{\tilde Y}_{m,j})\|^{p'} 1_{\{\boldsymbol Y_{m,j}\neq\boldsymbol{\tilde Y}_{m,j}\mathrm{~for~some~}j=1,...,i\}}\Big].
\end{eqnarray*}
Notice that
\begin{eqnarray*}
1_{\{\boldsymbol Y_{m,j}\neq\boldsymbol{\tilde Y}_{m,j}\mathrm{\mathrm{~for~some~}}j=1,...,i\}}\le 1_{\{\boldsymbol Y_{m,j}\neq\boldsymbol{\tilde Y}_{m,j}\mathrm{~for~some~}j=1,...,r\}}\quad\mathrm{for}~i\le r.
\end{eqnarray*}
Then we have
\begin{eqnarray*}
&&\E\Big[\max_{1\le i\le \kappa(2^{m+1})}\|\sum_{j=1}^{i}(\boldsymbol Y_{m,j}-\boldsymbol{\tilde Y}_{m,j})\|^{p'}1_{\{\boldsymbol Y_{m,j}\neq\boldsymbol{\tilde Y}_{m,j}\mathrm{\mathrm{~for~some~}}j=1,...,i \}}\Big]\\
&\le&\E\Big[\max_{1\le i\le \kappa(2^{m+1})}\|\sum_{j=1}^{i}(\boldsymbol Y_{m,j}-\boldsymbol{\tilde Y}_{m,j})\|^{p'}1_{\{\boldsymbol Y_{m,j}\neq\boldsymbol{\tilde Y}_{m,j}\mathrm{~for~some~}j=1,...,\kappa(2^{m+1})\}}\Big]\\
&\le&\Big\{\E\Big[\max_{1\le i\le \kappa(2^{m+1})}\|\sum_{j=1}^{i}(\boldsymbol Y_{m,j}-\boldsymbol{\tilde Y}_{m,j})\|^{2p'}\Big]\Big\}^{\frac12} \Big\{\E[1_{\{\boldsymbol Y_{m,j}\neq\boldsymbol{\tilde Y}_{m,j}\mathrm{~for~some~}j=1,...,\kappa(2^{m+1}) \}}]\Big\}^{\frac12}.
\end{eqnarray*}
A straight calculation yields
\begin{eqnarray*}
\E\Big[\max_{1\le i\le \kappa(2^{m+1})}\|\sum_{j=1}^{i}(\boldsymbol Y_{m,j}-\boldsymbol{\tilde Y}_{m,j})\|^{2p'}\Big]
&\le& \E\Big[\sum_{j=1}^{\kappa(2^{m+1})}\|\boldsymbol Y_{m,j}-\boldsymbol{\tilde Y}_{m,j}\|^{2p'}\Big]\\
&\le& C_{p'}\kappa(2^{m+1})^{2p'-1}\sum_{j=1}^{\kappa(2^{m+1})}(\E\|\boldsymbol Y_{m,j}\|^{2p'}+\E\|\boldsymbol{\tilde Y}_{m,j}\|^{2p'})\\
&\le& C_{p'}\kappa(2^{m+1})^{2p'}m_1^{p'},
\end{eqnarray*}
where the last line follows Lemma \ref{lem:Y} and the fact $\boldsymbol Y_{m,j}\overset{\mathscr{D}}{=}\boldsymbol{\tilde Y}_{m,j}$. For the indicator function part, we have
\begin{eqnarray*}
\E[1_{\{\boldsymbol Y_{m,j}\neq\boldsymbol{\tilde Y}_{m,j}\mathrm{~for~some~}j=1,...,\kappa(2^{m+1}) \}}]
&=&\PP(\boldsymbol Y_{m,j}\neq\boldsymbol{\tilde Y}_{m,j}\mathrm{~for~some~}j=1,...,\kappa(2^{m+1}) )\\
&\le& C\kappa(2^{m+1}) e^{-\beta m_2}.
\end{eqnarray*}
Thus,
\begin{eqnarray*}
\E\Big[\max_{1\le i\le \kappa(2^{m+1})}\|\sum_{j=1}^{i}(\boldsymbol Y_{m,j}-\boldsymbol{\tilde Y}_{m,j})\|^{p'}1_{\{\boldsymbol Y_{m,j}\neq\boldsymbol{\tilde Y}_{m,j}\mathrm{~for~some~}j=1,...,i\}}\Big]\le C m_1^{\frac {p'}2}\kappa(2^{m+1})^{p'+\frac12}e^{-\frac {\beta} 2 m_2}.
\end{eqnarray*}
Taking $x=2^{((1-\alpha_1)/p'+\alpha_1/2)m}\sqrt{\log 2^m}$ and $C^*\ge 2(1-\alpha_1)(p'-\frac12)/\beta$, we can get
\begin{eqnarray*}
&&\sum_{m=1}^\infty \PP\left(\max_{1\le i\le \kappa(2^{m+1})}\|\sum_{j=1}^{i}(\boldsymbol Y_{m,j}-\boldsymbol{\tilde Y}_{m,j})\|\ge x\right)\\
&\le& \sum_{m=1}^\infty Cx^{-p'} 2^{\frac {p'}2\alpha_1 m} 2^{(p'+\frac12)(1-\alpha_1)m}\exp\{-\frac{\beta}{2}m_2\}\\
&\le& \sum_{m=1}^\infty C (\log 2^m)^{-\frac {p'}2} 2^{(p'-\frac12)(1-\alpha_1)m}\exp\{-\frac{\beta}{2}C^*\log 2^m\}
<\infty.
\end{eqnarray*}
  The Borel-Cantelli lemma yields
\begin{eqnarray*}
\max_{1\le i\le \kappa(2^{m+1})}\|\sum_{j=1}^{i}(\boldsymbol Y_{m,j}-\boldsymbol{\tilde Y}_{m,j})\|=o(2^{((1-\alpha_1)/p'+\alpha_1/2)m}\sqrt{\log 2^m}),\quad \text{a.s.}.
\end{eqnarray*}

For any $j\le \kappa(2^{m+1})$, \eqref{e:bigblock1} immediately implies \eqref{e:bigblock2}.
\end{proof}

\begin{proof}[Proof of Lemma \ref{lem:smallblock}]
Similar with the proof of Lemma \ref{lem:bigblock}, we can construct independent random variables $\boldsymbol {\tilde Z}_{m,j}$ distributed as $\boldsymbol Z_{m,j}$ for $j=1,...,1+\kappa(2^{m+1})$ on a richer probability space and
\begin{eqnarray*}
	\PP(\boldsymbol Z_{m,j}\neq\boldsymbol {\tilde Z}_{m,j}\text{~~for~some~}1\le j\le \kappa(2^{m+1})+1)&\le& C\kappa(2^{m+1})\beta(m_1).
\end{eqnarray*}
Considering the following probability,
\begin{eqnarray*}
&&\PP\Big(\max_{2^m+1\le i\le n}\big\|\sum_{\ell\in\mathcal{J}(m)\cap[2^{m+1},i]}  \boldsymbol{X_\ell}\big\|\ge x\Big)\nonumber\\
&\le& \PP\Big(\max_{2^m+1\le i\le n}\big\|\sum_{j=1}^{\kappa(i)}(\boldsymbol{Z}_{m,j}-\boldsymbol {\tilde Z}_{m,j})+\sum_{j=\kappa(i)(m_1+m_2)+m_1+1}^{\left(\kappa(i)(m_1+m_2)+m_1\right)\vee i} \boldsymbol{X}_{j}\big\|+ \max_{2^m+1\le i\le n}\|\sum_{j=1}^{\kappa(i)}\boldsymbol {\tilde Z}_{m,j}\|\ge x\Big)\nonumber\\
&\le& \PP\Big(\max_{2^m+1\le i\le n}\|\sum_{j=1}^{\kappa(i)}\boldsymbol {\tilde Z}_{m,j}\|\ge x/2\Big)\\
&&+C_{p'} x^{-p'}\Big(\E\big[\max_{2^m+1\le i\le n}\|\sum_{j=1}^{\kappa(i)}(\boldsymbol Z_{m,j}-\boldsymbol {\tilde Z}_{m,j})\|^{p'}\big]
+\E\big[\max_{2^m+1\le i\le n}\|\sum_{j=\kappa(i)(m_1+m_2)+m_1+1}^{(\kappa(i)(m_1+m_2)+m_1)\vee i}  \boldsymbol{X}_j\|^{p'}\big]\Big).\nonumber
\end{eqnarray*}

For the first term, since $\boldsymbol {\boldsymbol {\tilde Z}}_{m,j}$ are centered i.i.d. random vectors, L\'{e}vy's inequality and Lemma \ref{lem:Y} imply
\begin{eqnarray*}
&&\PP\Big(\max_{2^m+1\le i\le n}\|\sum_{j=1}^{\kappa(i)}\boldsymbol {\tilde Z}_{m,j}\|\ge x/2\Big)\\
&\le& 2\PP\big(\|\sum_{j=1}^{\kappa(n)}\boldsymbol {\tilde Z}_{m,j}\|\ge x/2\big)
\le C x^{-p'}\E\|\sum_{j=1}^{\kappa(n)}\boldsymbol {\tilde Z}_{m,j}\|^{p'}\le Cx^{-p'}\kappa(2^{m+1})^{\frac {p'}2}m_2^{\frac {p'}2}.
\end{eqnarray*}
Taking $x=2^{m(1-\alpha_1)/2}\log 2^m$, one has
\begin{eqnarray}\label{e:smallb1}
\sum_{m=1}^\infty\PP\Big(\max_{2^m+1\le i\le n}\|\sum_{j=1}^{\kappa(i)}\boldsymbol {\tilde Z}_{m,j}\|\ge x/2\Big)<\infty.
\end{eqnarray}
For the second term, similar with the estimate of Lemma \ref{lem:bigblock}, a straight calculation implies
\begin{eqnarray*}
\E\big[\max_{2^m+1\le i\le n }\|\sum_{j=1}^{\kappa(i)}(\boldsymbol{Z}_{m,j}-\boldsymbol {\tilde Z}_{m,j})\|^{p'}\big] 
&=&\E\big[\max_{1\le i\le \kappa(n) }\|\sum_{j=1}^{i}(\boldsymbol{Z}_{m,j}-\boldsymbol {\tilde Z}_{m,j})\|^{p'}1_{\{\boldsymbol{Z}_{m,j}\neq\boldsymbol {\tilde Z}_{m,j} \mathrm{~for ~some~} j=1,...i\}}\big] \nonumber\\
&\le&C m_2^{\frac {p'}2}\kappa(2^{m+1})^{p'+\frac12}e^{-\frac{\beta}{2} m_1}.
\end{eqnarray*}
For the last term, following \citet[Proposition 1]{Wu07} and Lemma \ref{lem:Y}, one has
\begin{eqnarray*}
	\Big(\E[\max_{1\le i\le 2^r }\|\sum_{j=1}^i  \boldsymbol X_j\|^{p'}]\Big)^{\frac1{p'}}&\le&\sum_{i=0}^r2^{(r-i)/p'}\Big(\E\|\sum_{j=1}^{2^i}\boldsymbol  X_j\|^{p'}\Big)^{\frac1{p'}}\\
	&\le& C\sum_{i=0}^r2^{(r-i)/p'}(2^i)^{\frac12}\le C2^{\frac r2},
\end{eqnarray*}
which yields
\begin{eqnarray*}
\E\big[\max_{1\le i\le m_2}\|\sum_{j=1}^{i}  \boldsymbol X_j\|^{p'}\big]\le  m_2^{p'/2}.
\end{eqnarray*}
Then we have
\begin{eqnarray*}
&&\E\big[\max_{2^m+1\le i\le n}\|\sum_{j=\kappa(i)(m_1+m_2)+m_1+1}^{(\kappa(i)(m_1+m_2)+m_1)\vee i}  \boldsymbol X_j\|^{p'}\big]\\
&\le&\kappa(2^{m+1})\E\big[\max_{1\le i\le m_2}\|\sum_{j=1}^{i}  \boldsymbol X_j\|^{p'}\big]\le  \kappa(2^{m+1})m_2^{\frac{p'}2}
\end{eqnarray*}
Thus,
\begin{eqnarray}\label{e:smallb2}
&&\sum_{m=1}^\infty x^{-p'}C_{p'}\Big(\E\big[\max_{2^m+1\le i\le n}\|\sum_{j=1}^{\kappa(i)}(\boldsymbol Z_{m,j}-\boldsymbol {\tilde Z}_{m,j})\|^{p'}\big]
+\E\big[\max_{1\le i\le m_2}\|\sum_{j=1}^{i}  \boldsymbol X_j\|^{p'}\big]\Big)\\
&\le&\sum_{m=1}^\infty x^{-p'} C_{p'}\big( m_2^{\frac {p'}2}\kappa(2^{m+1})^{p'+\frac12}e^{-\frac{\beta}{2} m_1}  +\kappa(2^{m+1})m_2^{p'/2}\big)<\infty.\nonumber
\end{eqnarray}
Combining \eqref{e:smallb1} and \eqref{e:smallb2}, the Borel-Cantelli lemma implies
$$
\max_{2^m+1\le i\le n}\big\|\sum_{\ell\in\mathcal{J}(m)\cap[2^{m+1},i]} \boldsymbol X_\ell\big\|=o(2^{m(1-\alpha_1)/2}\log 2^m), \quad \text{a.s..}
$$

Similarly estimate holds for any i.i.d. centered Gaussian random vectors $\eta_\ell$.
\end{proof}

\section{ Proof of examples in Section \ref{sec:example}}\label{appen:ex}
\subsection{Proof of Example \ref{ex:1}}
\begin{lem}\label{lem:ergodic}
	Under Assumption {\bf (A3)}, $(\boldsymbol X_k)_{k\ge0}$ is exponential ergodic with invariant measure $\pi$, that is,
	\begin{eqnarray*}
		\sup_{|f|\le 1+V}\big|\E[f(\boldsymbol X_n)|\boldsymbol X_0=\boldsymbol x]-\pi(f)\big|\le V(\boldsymbol x)e^{-C_\gamma n},
	\end{eqnarray*}
	where $C_\gamma>0$ depends on $\gamma$.
\end{lem}
\begin{proof}
	We give the proof of the ergodicity of $(\boldsymbol X_n)_{n\ge0}$ following \citet[Theorem 2.1]{Tweedie94}. To verify condition $(7)$ of \citet{Tweedie94}, let
	\begin{align*}
	V^n(\boldsymbol x)=e^{C_\gamma n}V(\boldsymbol x)\quad r(n)= C_\gamma e^{C_\gamma n},
	\end{align*}
	the constant $C_\gamma>0$ will be chosen later. Following equation (\ref{e:lya}), one has
	\begin{eqnarray*}
		&&PV^{n+1}(\boldsymbol x)+r(n)V(\boldsymbol x)\\
		&\le& e^{C_\gamma (n+1)}\left(\gamma V(\boldsymbol x)+K\right)+ C_\gamma e^{C_\gamma n}V(\boldsymbol x)\\
		&=& e^{C_\gamma n}V(\boldsymbol x)+  C_\gamma e^{C_\gamma n}\left((\frac{C_\gamma-1}{C_\gamma}+\frac 1{C_\gamma}\gamma e^{C_\gamma})V(\boldsymbol x)+\frac 1{C_\gamma}  K e^{C_\gamma}\right).
	\end{eqnarray*}
	Choosing $C_\gamma$ small enough such that $\gamma e^{C_\gamma}-1+C_\gamma<0$, we can get
	\begin{eqnarray*}
		PV^{n+1}(\boldsymbol x)+r(n)V(\boldsymbol x)
		&\le& V^n(\boldsymbol x)+br(n)1_{\{\boldsymbol x\in \mathcal{C}\}},
	\end{eqnarray*}
	where $b=K  e^{C_\gamma}/C_\gamma$ and $\mathcal{C}=\{\boldsymbol x:V(\boldsymbol x)\le \frac{ K e^{C_\gamma}}{1-C_\gamma-\gamma e^{C_\gamma}}\}$. We deduce that $(\boldsymbol X_n)_{n\ge0}$ is ergodic with invariant measure $\pi$.	
\end{proof}
Recall the definition of $\beta-$mixing in \citet[Proposition 1]{Davy1973} which is equivalent with \eqref{e:mixing}.

\begin{definition}
	The $\beta-$mixing coefficients are given by:
	\begin{eqnarray*}
		\beta(n)=\int\sup_{0\le f\le1}|P_nf(x)-\int f\dif Q|\dif Q,
	\end{eqnarray*}
	where $Q$ is a stationary distribution. The process $\boldsymbol X_n$ is $\beta-$mixing if $\lim_{n\to\infty}\beta(n)=0$; is $\beta-$mixing with exponential
	decay rate if $\beta(n)\le Ce^{-cn}$ for some $C>0$ and $c>0$.
\end{definition}

\begin{lemma}\label{lem:mixing}
	Under Assumption {\bf (A3)}, one has
	\begin{eqnarray}\label{e:beta}
	\beta(n)&\le&\pi(V)e^{-C_\gamma n}.
	\end{eqnarray}
\end{lemma}
\begin{proof}
	Acording to Lemma \ref{lem:ergodic} and $\pi(V)< \infty$, one has
	\begin{eqnarray*}
		\beta(n)&=&\int\sup_{0\le f\le1}|\E[f(\boldsymbol X_n)|\boldsymbol X_0=\boldsymbol x]-\pi(f)|\pi(\dif \boldsymbol x)\nonumber\\
		&\le&{ \int\sup_{0\le f\le V}|\E[f(\boldsymbol X_n)|\boldsymbol X_0=\boldsymbol x]-\pi(f)|\pi(\dif \boldsymbol x)\nonumber}\\
		&\le&{ \int V(\boldsymbol x)e^{-C_\gamma n}\pi(\dif \boldsymbol x)\nonumber}\\
		&=& \pi(V)e^{-C_\gamma n}.
	\end{eqnarray*}
	Thus $(\boldsymbol X_n)_{n\ge0}$ is $\beta-$mixing with exponential decay rate.
\end{proof}
Thus condition ${\bf (A1)}$ is satisfied and $(\boldsymbol  X_n)_{n\ge0}$ satisfies ASIP with rate greater than $1/4+1/(4p-4)$.

\subsection{Proof of Example \ref{ex:3}}
It is easy to calculate that \eqref{e:AR1} has a unique stationary solution given by 
\begin{eqnarray}\label{e:AR1sol}
\boldsymbol X_k=\sum_{j=0}^{\infty}\boldsymbol A^j\boldsymbol B\boldsymbol \e_{k-j}.
\end{eqnarray}
For this exponential ergodic process, {\bf (A1)} is easy to verified. Now we verify the condition {\bf (A2)}.
\begin{eqnarray*}
	\cov(\sum_{k=0}^{n-1}\boldsymbol X_k)&=&\sum_{k=0}^{n-1}\cov(\boldsymbol X_k)+\sum_{0\le i<j\le n-1}(\cov(\boldsymbol X_i,\boldsymbol X_j)+\cov(\boldsymbol X_i,\boldsymbol X_j)^T)\\
	&=& n\cov(\boldsymbol X_0)+\sum_{0\le i<j\le n-1} (\cov(\boldsymbol X_0,\boldsymbol X_{j-i})+\cov(\boldsymbol X_0,\boldsymbol X_{j-i})^T).
\end{eqnarray*}
Following \eqref{e:AR1sol}, one has
\begin{eqnarray*}
\cov(\boldsymbol X_0)=\cov\big(\sum_{j=0}^{\infty}\boldsymbol A^j\boldsymbol B\boldsymbol \e_{-j} \big)=(\boldsymbol I-\boldsymbol A^2)^{-1}\boldsymbol B^2,
\end{eqnarray*}
and
\begin{eqnarray*}
	\cov(\boldsymbol X_0,\boldsymbol X_{j-i})&=&\cov\big(\sum_{r=0}^{\infty}\boldsymbol A^r\boldsymbol \e_{-r},\sum_{r=0}^{j-i-1}\boldsymbol A^r\boldsymbol \e_{-j-i-r}+\sum_{r=0}^{\infty}\boldsymbol  A^{j-i+r}\boldsymbol \e_{-r}\big)\\
	&=& \cov(\boldsymbol X_0)\boldsymbol A^{j-i}.
\end{eqnarray*}
Thus, we have as $n\to\infty$,
\begin{eqnarray*}
	\cov(\sum_{k=0}^{n-1}\boldsymbol X_k)/n&=&\cov(\boldsymbol X_0)+\frac1n\sum_{k=1}^{n-1}k\big(\cov(\boldsymbol X_0,X_{n-k})+\cov(X_0,\boldsymbol X_{n-k})^T\big)\\
	&\to&  \cov(\boldsymbol X_0)+2\cov(\boldsymbol X_0)(\boldsymbol A^{-1}-\boldsymbol I)^{-1},
\end{eqnarray*}
i.e., $\boldsymbol\Gamma=\cov(\boldsymbol X_0)+2\cov(\boldsymbol X_0)(\boldsymbol A^{-1}-\boldsymbol I)^{-1}$.
It is easy to see that,
$$
\boldsymbol\Gamma \boldsymbol e_i(x)=\big(\frac{\lambda_i}{1-\lambda_i^2}+\frac{2\lambda_i^2}{(1-\lambda_i^2)(1-\lambda_i)}\big)\boldsymbol e_i(x).
$$
Thus {\bf (A2)} is satisfied by taking $\lambda_i \asymp i^{-\delta}$ and ASIP holds.


\end{appendix}
\ \\

{\bf Acknowledgements}: L. Xu is supported in part by NSFC Grant No.12071499, Macao S.A.R. grant FDCT  0090/2019/A2 and University of Macau grant  MYRG2020-00039-FST.

\bibliographystyle{plainnat}

\end{document}